\documentclass[11pt]{article}

\usepackage{tikz}
\tikzstyle{vertex}=[circle, draw, inner sep=0pt, minimum size=5pt]

\usetikzlibrary{decorations.markings}

\usepackage[english]{babel}
\usepackage{times}
\usepackage{amsthm}
\usepackage{amssymb, verbatim}
\usepackage{amsmath}
\usepackage{eucal}
\usepackage{mathrsfs}

\usepackage{fdsymbol}

\newtheorem{theorem}{Theorem}[section]
\newtheorem{corollary}[theorem]{Corollary}
\newtheorem{lemma}[theorem]{Lemma}
\newtheorem{proposition}[theorem]{Proposition}

\theoremstyle{definition}

\newtheorem{remark}[theorem]{Remark}
\newtheorem{da-scrivere}[theorem]{da-scrivere}

\newcommand{\Aut}{{\mathrm {Aut}}}
\newcommand{\out}{{\mathrm {Out}}}

\newcommand{\End}{\operatorname{End}}

\def\B{{\cal B}}

\def\D{{\cal D}}

\def\F{{\cal F}}

\def\H{{\cal H}}
\def\K{{\cal K}}

\def\O{{\cal O}}
\def\P{{\cal P}}
\def\Q{{\cal Q}}
\def\S{{\cal S}}

\def\U{{\cal U}}

\begin{document}
\title{A look at the inner structure of the $2$-adic ring $C^*$-algebra and its automorphism groups}
\author{Valeriano Aiello$^\dag$, Roberto Conti$^\sharp$, Stefano Rossi$^\natural$\footnote{E-mail: 
valerianoaiello@gmail.com, roberto.conti@sbai.uniroma1.it, rossis@mat.uniroma2.it}
\\
$^\dag$ Dipartimento di Matematica e Fisica\\ Universit\`a Roma Tre \\ Largo S. Leonardo Murialdo 1, 00146 Roma, Italy.\\
$^\sharp$ Dipartimento di Scienze di Base e Applicate per l'Ingegneria \\ Sapienza Universit\`a di Roma \\ Via A. Scarpa 16,
I-00161 Roma, Italy.\\
$^\natural$ Dipartimento di Matematica,\\ Universit\`a di Roma Tor
Vergata \\ Via della Ricerca Scientifica 1, I--00133 Roma, Italy.}
\date{}
\maketitle

\begin{abstract}
We undertake a systematic study of the so-called  $2$-adic ring $C^*$-algebra $\Q_2$. This is the universal $C^*$-algebra generated by a unitary $U$ and an isometry $S_2$ such that $S_2U=U^2S_2$ and $S_2S_2^*+US_2S_2^*U^*=1$. 
Notably, it contains a copy of the Cuntz algebra $\O_2=C^*(S_1, S_2)$ through the injective homomorphism mapping $S_1$ to $US_2$.
Among the main results, the relative commutant $C^*(S_2)'\cap \Q_2$ is shown to be trivial. This  in turn leads to a rigidity property enjoyed by the inclusion $\O_2\subset\Q_2$, namely the endomorphisms of $\Q_2$ that restrict to the identity on $\O_2$ are actually the identity on the whole 
$\Q_2$. Moreover, there is no conditional expectation from $\Q_2$ onto $\O_2$. As for the inner structure of $\Q_2$,  the diagonal subalgebra $\D_2$ and $C^*(U)$ are both proved to be maximal abelian in $\Q_2$. The maximality of the latter allows a thorough investigation of several classes of endomorphisms and automorphisms of $\Q_2$. In particular, the semigroup of the endomorphisms fixing $U$ turns out to be a maximal abelian subgroup of $\Aut(\Q_2)$ topologically isomorphic with $C(\mathbb{T},\mathbb{T})$. Finally, it is shown 
by an explicit construction
that ${\rm Out}(\Q_2)$ is uncountable and non-abelian.
\end{abstract}

\tableofcontents

\section{Introduction}
Ever since their formal debut in the most cited paper \cite{Cuntz1}, the Cuntz algebras have received a great deal of attention. The reasons are so many they resist any attempt to be only briefly accounted, and this introduction will be no exception.
Therefore, we cannot but draw a rather quick and incomplete outline of the later developments until the present state of the art,  if only to better frame the scope of our work. 
For many authors who have focused their interest on more and more general constructions inspired by the Cuntz algebras,
there are as many authors who have devoted themselves to as thorough as possible a study of the concrete Cuntz algebras. 
This study includes, in particular, an in-depth investigation of endomorphisms and automorphisms. Cuntz is  among those who have undertaken both the tasks.   
As for the first, he and other authors have written a long series of works where increasingly broad classes of $C^*$-algebras associated with algebraic objects such as rings are contrived. 
In particular, in \cite{Cuntz3} he introduced a $C^*$-algebra $\Q_{\mathbb{N}}$ associated with the $ax+b$-semigroup over the natural numbers. A few years later, Larsen and Li \cite{LarsenLi} considered  its $2$-adic version which, accordingly, they denoted by $\Q_2$. The main object of our interest in the present paper, this novel
 $C^*$-algebra is in fact naturally associated with the semidirect
product semigroup of the additive group $\mathbb{Z}$ acted upon by multiplication with non-negative powers of $2$.
It did appear before elsewhere, cf. \cite{LarsenLi} and the references therein, but it is in the above-mentioned work of Larsen and Li that it was studied systematically for the first time. 
After recalling that $\Q_2$ is a nuclear $C^*$-algebra, they prove, among other things, that $\Q_2$ is also a purely infinite simple $C^*$-algebra. 
They give two proofs of this fact. Notably, one is a straightforward  application of $\Q_2$ being a Cuntz-Pimsner algebra, to which general results of Exel, an Huef and Raeburn \cite{Exel} apply.
From our viewpoint, this lucky circumstance is well worth mentioning. Indeed, very little is known about the general structure of endomorphisms or automorphisms for general Pimsner algebras, cf. \cite{Zac,DS,CHSJFA}.
Therefore, as should follow from some of the main results announced in the abstract, a good way to look at $\Q_2$ might be to regard it as a felicitous example of a Pimsner algebra for which a far-reaching study is not that prohibitive.
Far be it from us, however, to allege we have done all that could be done. Rather, our hope is that further research may stem from this work. 
That for a Pimsner algebra a thorough comprehension of the properties of all its endomorphisms is a virtually impossible task should be no surprise. Indeed, already for the Cuntz algebras $\O_n$ the problem, at least in its full generality, has turned out to be well beyond the reach of current research. For instance, all endomorphisms of $\O_n$ are known to come from unitary elements of $\O_n$ via a correspondence first pointed out by Takesaki, see e.g. \cite{Cuntzsurvey}, and yet it is notoriously difficult to find non-tautological and effective characterizations of those unitaries that yield automorphisms, 
although a number of intriguing if partial results about the structure of $\Aut(\O_n)$ have recently been achieved in \cite{ContiSzymanski,CKS-MoreLoc, Roberto, CRS, CHSCrelle,CHS,CHS2015}, see also \cite{ContiHongSzymanski, Contisurvey} for an informative account. 
Without further ado, we can now move on to the basic definitions needed throughout the paper. While being a Pimsner algebra, the $2$-adic ring $C^*$-algebra is perhaps best described as the universal $C^*$-algebra 
$\Q_2$ generated by a unitary $U$ and an isometry $S_2$ such that 
$$
S_2U=U^2S_2\quad\textrm{and}\quad S_2S_2^*+US_2S_2^*U^*=1
$$
The reader interested in its description in terms of Pimsner algebras is again referred to 
\cite{LarsenLi}. However, in this paper we will never need to resort to that picture, which is why we may as well dispense with it. What we do need to observe is that $\Q_2$ contains 
a copy of the Cuntz algebra $\O_2$. Indeed, the latter is by definition the universal $C^*$-algebra generated by two isometries $X_1$ and $X_2$ such that $X_1X_1^*+X_2X_2^*=1$. Therefore, the map
taking $X_1$ to $US_2$ and $X_2$ to $S_2$ extends by universality to a homomorphism from $\O_2$ to $\Q_2$, which is injective thanks to the simplicity of $\O_2$.  Accordingly, as of now it will be
convenient to think of $\O_2$ as being a subalgebra of $\Q_2$. To us the rather explicit description of the inclusion $\O_2\subset\Q_2$ was in fact among the strongest motivations to carry out the present study of $\Q_2$, especially as far as the extension problem is concerned. This asks whether an endomorphism of $\O_2$ extends to $\Q_2$.  It turns out that this is not always the case. For instance, as soon as Bogoljubov automorphisms are looked at, 
easy examples are found of non-extensible automorphisms. More precisely, we find that the only extensible Bogoljubov automorphisms are the flip-flop, the gauge automorphisms and their products.
In addition, facing the extension problem in general leads to an interesting rigidity property enjoyed by the inclusion $\O_2\subset\Q_2$, namely if an endomorphism of $\Lambda$ of $\Q_2$ restricts to $\O_2$ trivially, then it is the identity automorphism. To the best of our knowledge, the pair $(\O_2,\Q_2)$ is the only known example of a non-trivial inclusion of Pimsner algebras that fulfills the rigidity condition. In this respect, it is also worth mentioning that there is no conditional expectation from the larger onto the smaller of the two.\\ 

Once these questions have been answered, it is natural to go on to study endomorphisms and automorphisms of $\Q_2$ irrespective of whether they leave $\O_2$ globally invariant or not. 
Asking questions of this sort is of course motivated by the overwhelming literature written on similar issues for the Cuntz algebras. However, this entails a preliminary study of the inner structure of $\Q_2$. In this regard, we prove that both the $C^*$-algebra generated by $U$ and the diagonal subalgebra $\D_2\subset\O_2$ are maximal abelian. It came as a surprise to us to learn that not as many results as one would expect are known on maximal abelian subalgebras for general $C^*$-algebras. Apparently, that of maximal abelian subalgebra is a notion far more relevant to von Neumann algebras. Yet $\Q_2$ seems to be one of the few exceptions, for its theory does benefit from $C^*(U)$ being such a subalgebra. Indeed, we exploit the maximality of $C^*(U)$ to derive a number of results on the general form of selected classes of automorphisms, many of which are, incidentally,  quasi-free in the sense of Dykema-Shlyakhtenko and Zacharias, see \cite{DS, Zac}. Notably, we show that the semigroup of the endomorphisms of $\Q_2$ that fix $U$ is in fact a maximal abelian subgroup of $\Aut(\Q_2)$ isomorphic with $C(\mathbb{T},\mathbb{T})$, the group of all continuous $\mathbb{T}$-valued functions defined on the one-dimensional torus $\mathbb{T}$ understood as the spectrum of $U$. These results are in fact in same spirit as those expounded in \cite{Cuntzsurvey}. 
Moreover, they indicate that it is not unconceivable to regard $C(\mathbb{T},\mathbb{T})\subset\Aut(\Q_2)$ as a generalized maximal torus, although not connected, for which a kind of infinite-dimensional Weyl theory might well worth attempting, as is done in \cite{CHSCrelle, CHS} for Cuntz algebras.\\

Since the group of equivalence classes of outer automorphisms of $\O_2$ is known to be so large as to contain most locally compact groups, our investigation also addresses $\out(\Q_2)$.  We have partial evidence to hold that $\out(\Q_2)$ is not as large, not least because, as a drawback of the aforementioned result on Bogoljubov automorphisms, we no longer have a general procedure for embedding locally compact groups into $\out(\Q_2)$ as we  would do with $\out(\O_2)$. At any rate, we prove that $\out(\Q_2)$ is still an uncountable non-abelian group. This is done in two steps. First, we prove that both the flip-flop and the gauge automorphisms are mutually non-equivalent outer automorphisms. Second, we provide a broad class of outer automorphisms that do not commute in $\out(\Q_2)$ with  the flip-flop. Even so, the non-commutativity thus exhibited is admittedly of a rather mild form. We do believe that it is an interesting, albeit difficult, problem to say to what extent $\out(\Q_2)$ is non-abelian.\\

A few words on the organization of the material are in order.  The various results of the paper are scattered throughout several sections, which more or less follow the order in which the topics developed have been introduced above, as to allow the reader to find them more easily. 
For convenience, here follows a description of the content of the several sections, to be also understood as a short guide to the main results.
Section \ref{SecPrep} is preparatory in character, as it sets the stage for our subsequent considerations. Indeed, all the needed definitions and  basic properties are to be found here.    
In Section \ref{SecMASA} both $C^*(U)$ and $\D_2$ are shown to be maximal abelian subalgebras of $\Q_2$, see Theorems \ref{UMASA} and \ref{DMASA}, respectively.
In Section \ref{SecIrr} the extended canonical endomorphism is proved to be a shift on $\Q_2$, Theorem \ref{Shift}. Moreover, there is no conditional expectation from $\Q_2$ onto $\O_2$,
Theorem \ref{Nocondexp}. The main result of Section \ref{Secrelcomm} is Theorem \ref{Relcomm}, where the relative commutant
$C^*(S_2)'\cap\Q_2$ is shown to be trivial. Section \ref{SecExt} is focused on the uniqueness of extensions of automorphisms from $\O_2$ to $\Q_2$, which is
proved in Theorem \ref{Rigidity}, and the non-extendability of general Bogoljubov automorphisms, which is proved
in Theorem \ref{Bogoljubov}. Section \ref{Outerness} deals with the outerness of the gauge automorphisms and the flip-flop, but it also includes a general result, see Theorems \ref{GaugeOut}, \ref{FlipFlopOut}, \ref{GenOut},
respectively. Finally, Section \ref{SecAut} provides a complete description of $\Aut_{C^*(U)}(\Q_2)$, see Theorem \ref{Struct}. Moreover, this group is shown to be isomorphic
with $C(\mathbb{T},\mathbb{T})$ and maximal abelian in $\Aut(\Q_2)$, see Theorems \ref{Isom} and \ref{MaxAb}, respectively.   
The  last two results along with Theorem \ref{out-not-ab}, which states that the outer automorphism group is non-abelian, should be regarded as the main results of the present paper. \\

Throughout the paper, all endomorphisms are assumed to be unital and $*$-preserving.
Finally, for endomorphisms of the Cuntz algebra $\O_2$  we adopt the well-established notations to be found 
in the wide literature of the field (see the beginning of Section \ref{SecExt} for a very short description of the Cuntz-Takesaki correspondence).
This is certainly the case for the symbols $\alpha_z$, $\varphi$, $\lambda_f$ introduced in full detail in the next section.

\section{First results}\label{SecPrep}

As we observed in the introduction, the  $2$-adic ring $C^*$-algebra contains a copy of the Cuntz algebra $\O_2$, as the $C^*$-subalgebra generated by
$S_2$ and $S_1\doteq US_2$. Since the theory of the latter has been enriched by a deeper and deeper knowledge
of distinguished classes of endomorphisms as well as automorphisms, problems to do with their extensions to $\Q_2$ are undeniably among the most natural things to initiate a study with. We will see in Section 4.1 that as soon as too much generality is allowed,
these problems begin to be intractable for all practical purposes. If one asks a bit more specific questions, a great many partial  results do start cropping out. 
At any rate, we will also obtain a general result, namely that whenever extensions exist they are unique, which is yet another way to state the rigidity property of the inclusion $\O_2\subset\Q_2$ we explained in the introduction. The precise statement of this fact, too, is contained in Section 4.1.     
In the present section, we limit ourselves to three remarkable examples that are easily dealt with. The first is the canonical shift. The second is the flip-flop. The third are the gauge automorphisms.\\           

The canonical shift is explicitly defined on every $x\in\O_2$ as $\varphi(x)=S_1xS_1^*+S_2xS_2^*$, therefore if we set  
$$\widetilde\varphi(x) = US_2 x S_2^*U^* + S_2 x S_2^*\quad\textrm{for any}\; x \in {\mathcal Q}_2$$ 
we still define an endomorphism of $\Q_2$, which restricts to $\O_2$ as the usual shift. The intertwining rules $S_i x = \widetilde\varphi(x)S_i$ for any  $x\in {\mathcal Q}_2$ with $i=1,2$ still hold true.
Moreover, a straightforward computation shows that $\widetilde\varphi(U) = U^2 $.
Since the continuous functional calculus of a normal operator commutes with any endomorphism, the above equality can also be rewritten as $\widetilde\varphi(f(U))=f(U^2)$, which is true for any continuous function $f$.  
It goes without saying that the same equality retains its validity with any Borel function whenever $\Q_2$ is represented on some Hilbert space. We shall avail ourselves of this useful fact later on.\\

As is well known, the flip-flop is the involutive automorphism $\lambda_f\in\Aut(\Q_2)$ that switches $S_1$ and $S_2$ with each other. 
The flip-flop extends as well, although the proof is less obvious and needs an argument. This is done here below.
\begin{proposition}(\cite{CS-private})
The flip-flop automorphism of $\O_2$ extends to an automorphism of $\Q_2$. 
\end{proposition}
\begin{proof}
If we set $U'\doteq U^*$ and $S_2'\doteq US_2$, then the identity $S_2'S_2'^*+U'S_2'S_2'^*U^*=1$ is immediately checked. By universality of $\Q_2$, there exists a unique endomorphism $\widetilde\lambda_f\in\End(\Q_2)$ such that $\widetilde\lambda_f (U)=U'=U^*$ and $\widetilde\lambda_f(S_2)=S_2'=US_2$.
This endomorphism is necessarily injective as $\Q_2$ is simple.
 Since $\widetilde\lambda_f(U^*)=U$ and $\widetilde\lambda_f(US_2)=S_2$, the image of $\alpha$ must be the whole $\Q_2$, that is to say $\widetilde\lambda_f$ is an automorphism. Finally, it is obviously an extension of the flip-flop.
\end{proof}
As of now, the above extension will be referred to simply as the flip-flop of $\Q_2$ and will be denoted by $\widetilde\lambda_f$.\\

It is also well known that the Cuntz algebra $\O_2$ is acted upon by $\mathbb{T}$ through the so-called gauge automorphisms $\alpha_z$ given by $\alpha_z(S_i)=zS_i$, with $z\in\mathbb{T}$. We can also prove the following extension result concerning gauge automorphisms.
\begin{proposition}
The gauge automorphisms of $\O_2$ can all be extended to automorphisms of $\Q_2$.
\end{proposition}
\begin{proof}
Now we set $U'\doteq U$ and $S_2'\doteq zS_2$, where $z$ is any complex number of absolute value equal to $1$ . As we still have $S_2'S_2'^*+U'S_2'S_2'^*U^*=1$, there exists an automorphism $\widetilde\alpha_z\in\Aut(\Q_2)$ such that $\widetilde\alpha_z(U)=U$ and
$\widetilde\alpha_z(S_2)=zS_2$. To conclude, all that we are left to do is note that $\widetilde\alpha_z(S_1)=\widetilde\alpha_z(US_2)=\widetilde\alpha_z(U)\widetilde\alpha_z(S_2)=U zS_2=zS_1$.
\end{proof}
With a slight abuse of terminology, the automorphisms $\widetilde\alpha_z$ obtained above will be referred to as the gauge automorphisms. To conclude, it is worth noting that the flip-flop and the gauge automorphisms commute.\\

\subsection{The gauge-invariant subalgebra}

The gauge-invariant subalgebra of $\O_2$, usually denoted  by $\F_2$, is known to be isomorphic with the CAR algebra. 
The corresponding gauge-invariant subalgebra of $\Q_2$, which throughout this paper will be denoted by $\Q_2^{\mathbb{T}}$, can no longer be identified with such a remarkable $C^*$-algebra. However, it can be described far more conveniently as the closure of a suitable linear span. To do so, we need to point out the following simple but useful result. In order to state it as clearly as possible, let us first set some notation. As in \cite{Cuntz1},  we denote by $W_2$ the set of all multi-indices $\mu=(\mu_1, \mu_2,\ldots, \mu_n)$ with $\mu_i\in\{1, 2\}$ and $n\in\mathbb{N}$; the integer $n$ is commonly referred to as the length of the multi-index $\mu$ and is
denoted by $|\mu|$.  For any such multi-index $\mu=(\mu_1, \mu_2,\ldots, \mu_n)$, we denote by $S_\mu$ the monomial 
$S_{\mu_1}S_{\mu_2}\ldots S_{\mu_n}$. 
\begin{proposition}
$\Q_2=\overline{{\rm span}}\{S_\mu S_\nu^*U^k: \mu,\nu\in W_2, k\in\mathbb{Z}\}$.
\end{proposition}
\begin{proof}
In order to prove the equality above all we have to do is observe that the following relations allow us to take both $U$ and $U^*$
from the left to the right side of any monomial of the form $S_\mu\S_\nu^*$.
\begin{itemize}
\item $US_1=S_2U$
\item $US_2=S_1$
\item $US_1^*=S_2^*U$
\item $US_2^*=S_2^*U^2$
\item $U^*S_1=S_2$
\item $U^*S_2=S_1U^*$
\item $U^*S_1^*=S_1^*(U^*)^2$
\item $U^*S_2^*=S_1^*U^*$
\end{itemize}
The relations themselves are immediately verified by direct computation instead.
\end{proof}
The gauge automorphisms yield  a conditional expectation $\widetilde E$ from $\Q_2$ onto $\Q_2^{\mathbb{T}}$ by averaging the action itself on $\mathbb{T}$, that is for any $x\in\Q_2$ we have $\widetilde E(x)=\int_{\mathbb{T}}\widetilde\alpha_z(x){\rm d}z$, with ${\rm d}z$ being the normalized Haar measure of $\mathbb{T}$.
Now since $\widetilde E(S_\mu S_\nu^*U^k)=S_\mu S_\nu^*U^k\int_0^{2\pi}e^{i(|\mu|-|\nu|)\theta}\frac{{\rm d}\theta}{2\pi}$, we also have $\widetilde E(S_\mu S_\nu^*U^k)=0$ if and only if $|\mu|\neq |\nu|$. This helps to prove the description alluded to above.
\begin{proposition}
The equalities below hold:
$$\Q_2^{\mathbb{T}}=\overline{{\rm span}}\left\{S_\mu S_\nu^*U^k: \mu,\nu\in W_2, |\mu|=|\nu|, k\in\mathbb{Z}\right\}=C^*(U, \F_2)\subset\Q_2$$
\end{proposition}
\begin{proof}
The second equality is obvious. We focus then on the first, for which we only have to worry about the inclusion $\Q_2^{\mathbb{T}}\subset \overline{{\rm span}}\{S_\mu S_\nu^*U^k: |\mu|=|\nu|\}$, the other being immediately checked. If $x\in\Q_2^{\mathbb{T}}$, then $x=\widetilde E(x)$.  Now pick a sequence $\{x_n\}$ in the algebraic linear span of the set $\{S_\mu S_\nu^* U^k: \mu,\nu\in W_2, k\in\mathbb{Z}\}$ such that $\|x_n- x\|$ tends to zero. As $\tilde E$ is a bounded map, $\|\widetilde E(x_n)-\widetilde E(x)\|=\|\widetilde E(x_n)-x\|$ tends to zero as well. The conclusion follows easily now because $\widetilde{E}(x_n)\in {\rm span}\{S_\mu S_\nu^* U^k: |\mu|=|\nu|\}$ by the remark we made above.
\end{proof}
We should also mention that $C^*(\F_2,U)=C^*(\D_2,U)$ is the Bunce-Deddens algebra of type $2^\infty$, see \cite[Remark 2.8]{BOS}.

\subsection{The canonical representation}
In this section we gather as much information as we need about a distinguished representation of $\Q_2$, which will actually play a major role in most of what follows here and in the next sections. As far as we know, it was first  exhibited in \cite{LarsenLi}, where it is called the canonical representation. Therefore, from now on it will always be referred to as the canonical representation. After a brief review of its main properties, we discuss a number of results where the canonical representation proves to be rather useful.\\

The canonical representation acts on $\ell_2(\mathbb{Z})$ through the operators $S_2, U\in B(\ell_2(\mathbb{Z}))$ given by $S_2 e_k\doteq e_{2k}$ and $Ue_k\doteq e_{k+1}$, where $\{e_k: k\in\mathbb{Z}\}$ is the canonical orthonormal basis of $\ell_2(\mathbb{Z})$, i.e. $e_k(m)=\delta_{k,m}$. The very first thing to note is that $1$ is the only eigenvalue of $S_2$, corresponding to the one-dimensional eigenspace generated by $e_0$. 
This simple observation enables  us to give a short proof that the canonical representation is irreducible. Since we do not know of any reference where  this possibly known fact is explicitly pointed out, we do include an independent proof for the reader's convenience.
\begin{proposition}
The canonical representation of $\Q_2$ is irreducible.
\end{proposition}
\begin{proof}
Let $M\subset \ell_2(\mathbb{Z})$ be a $\Q_2$-invariant closed subspace. If $P$ is the associated orthogonal projection, then $P\in\Q_2'$. In particular, $S_2P=PS_2$, and so $S_2Pe_0=Pe_0$. As the eigenspace of $S_2$ corresponding to the eigenvalue $1$ is spanned by $e_0$, we must have either $Pe_0=e_0$ or $Pe_0=0$. In the first case, $e_0\in M$, and therefore $C^*(U)e_0\subset M$, which says that $M=\ell_2(\mathbb{Z})$ because $e_0$ is a cyclic vector for $C^*(U)$. In the second, $e_0\in M^{\perp}$ instead. As above, $M^{\perp}$ being $\Q_2$-invariant too, we have $M^{\perp}=\ell_2(\mathbb{Z})$, i.e. $M=0$. Note, however, that $\O_2$ does not act irreducibly on $\ell_2(\mathbb{Z})$, for the closed span of the set $\{e_k: k=0,1,\dots,\}$ is obviously a proper $\O_2$-invariant subspace. 
\end{proof}

\medskip

However, the canonical representation restricts to $\O_2$ as a reducible representation, which we denote by $\pi$. More precisely, it is a direct sum of two inequivalent irreducible representations of $\O_2$. To see this, 
let us define $\H_+, \H_-\subset \ell_2(\mathbb{Z})$ as the closed subspaces given by  $$\H_+\doteq \overline{{\rm span}\{e_k: k\geq 0\}}$$  and $$\H_-\doteq \overline{{\rm span}\{e_k: k< 0\}}$$
The Hilbert space $\ell_2(\mathbb{Z})$ is immediately seen to decompose into the direct sum of these subspaces, 
i.e.  $\ell_2(\mathbb{Z})=\H_+\oplus\H_-$. Furthermore, both $\H_+$ and $\H_-$ are $\O_2$-invariant, and finally they may be checked to be $\O_2$-irreducible too. This last statement should be a well-known fact. Even so, we give the proof for the sake of self-containedness. 
\begin{proposition}
The subspaces $\H_{\pm}$ are both $\O_2$-irreducible.
\end{proposition}
\begin{proof}
We only need to worry about $\H_+$, for $\H_-$ is dealt with in much the same way. Exactly as above, if $M\subset H_+$ is an $\O_2$-invariant subspace, then either $M$ or its orthogonal complement $M^{\perp}$ must contain $e_0$. The proof is thus complete if we can show that an $\O_2$-invariant subspace containing $e_0$, say $N$, is the whole $\H_+$, and this is proved once we show $e_k\in N$ for every $k\geq 0$. This is in turn easily achieved by induction on $k$. Suppose we have proved $\{e_l: l=0,1,\dots,k\}\subset N$. For the inductive step we have two cases, according as $k+1$ is even or odd. If it is even, then $e_{k+1}=S_2 e_{\frac{k+1}{2}}$; if it is odd, then $e_{k+1}=S_1e_{\frac{k}{2}}$. In either cases we see that $e_{k+1}$ is in $N$, as wished.
\end{proof}
Denoting by $\pi_{\pm}$ the restriction of $\pi$ to $\H_{\pm}$ respectively, the decomposition into irreducible representations $\pi=\pi_+\oplus\pi_-$ has just been proved to hold. Now,  as what we are really interested in is the commutant $\pi(\O_2)'$, we also need to observe that $\pi_+$ and $\pi_-$ are disjoint. This is done here below.
\begin{lemma}
If $\pi_+$ and $\pi_-$ are the irreducible representations defined above, then $\pi_+\downspoon \pi_-$.
\end{lemma}
\begin{proof}
It is enough to note that $\pi_+(S_2)$ has $1$ in its point spectrum, whereas $\pi_-(S_2)$ does not.
\end{proof}
To state the next result as clearly as possible some notation is needed, so let us  denote by $E_{\pm}$ the orthogonal projections onto $\H_{\pm}$ respectively.
\begin{proposition}
 The commutant of $\pi(\O_2)$ is given by  $\pi(\O_2)'=\mathbb{C}E_++\mathbb{C}E_-$. 
\end{proposition}
\begin{proof}
According to the decomposition described above we have $\pi(\O_2)'=(\pi_+(\O_2)\oplus\pi_-(\O_2))'$. But because $\pi_+$ and $\pi_-$ are disjoint, we can go a bit further and write
$$(\pi_+(\O_2)\oplus\pi_-(\O_2))'=(\pi_+(\O_2))'\oplus(\pi_-(\O_2))'=\mathbb{C}E_+\oplus\mathbb{C}E_-$$
where the last equality is due to the irreducibility of $\pi_{\pm}$. 
\end{proof}
This immediately leads to the following corollary, which needs no proof.
\begin{corollary}
In the canonical representation of $\Q_2$ the bicommutant of $\O_2$ is given by  $$\pi(O_2)''=\{T\in B(\ell_2(\mathbb{Z})): T=T_+\oplus T_-: T_{\pm}\in B(\H_{\pm})\}$$
\end{corollary}

The information we have gathered about $\pi$ will actually turn out to be vital in tackling the problem as to whether $C^*(U)'\cap \O_2\subset\Q_2$ is trivial. This is the case indeed, as anyone would expect. However, the relative proof is not as obvious as the statement. In fact, we still lack some basic ingredients. In particular, we need to observe that the basis vectors $e_k$ are all cyclic and separating for $U$. Therefore, 
the $W^*$-algebra $W^*(U)$ generated by $U$ is a maximal abelian von Neumann algebra of $B(\ell_2(\mathbb{Z}))$. Furthermore, it is common knowledge that 
$W^*(U)$ is  isomorphic with $L^{\infty}(\mathbb{T},\mu)$, where $\mu$ is the Haar measure of $\mathbb{T}$.\\ 

In passing, we also take the opportunity to exploit the canonical representation to show that $\D_2''\subset B(\ell_2(\mathbb{Z}))$ is a maximal abelian subalgebra as well. This will in turn be vital to conclude that $\D_2$ is a maximal abelian subalgebra of $\Q_2$.  Henceforward we shall denote by $\ell_\infty(\mathbb{Z})$ the atomic \emph{MASA} of $B(\ell_2(\mathbb{Z}))$ acting through diagonal operators with respect to the canonical basis. 

\begin{proposition}
In the canonical representation we have $\D_2'=\ell_\infty(\mathbb{Z})$.
\end{proposition}  
\begin{proof}
Since $\ell_\infty(\mathbb{Z})$ is a MASA, it is enough to prove that $\D_2''=\ell_\infty(\mathbb{Z})$, which will be immediately checked once we have proved that the projections $E_k$ onto $\mathbb{C}e_k$ all
belong to the strong closure of $\D_2$. To begin with, we note that the sequence $\{S_2^n(S_2^*)^n\}\subset\D_2$ strongly converges to $E_0$. But then the sequence $\{U^kS_2^n(S_2^*)^nU^{-k}:n\in\mathbb{N}\}$ strongly converges
to $E_k$. The conclusion now follows from the fact that $\D_2$ is globally invariant under $\rm {ad}(U)$.
\end{proof}

We now have all the necessary tools to get to prove that  $C^*(U)'\cap\O_2 $ is trivial.
This will in turn  result from a straightforward application of the next proposition, where much more is proved.
\begin{proposition}\label{Intr}
We have $W^*(U)\vee\pi(\O_2)'=B(\ell_2(\mathbb{Z}))$.
\end{proposition}
\begin{proof}
Let $P$ be an orthogonal projection in the commutant  of $W^*(U)\vee\pi(\O_2)'$. From $PE_+=E_+P$ and $PE_-=E_-P$, $H_\pm$ are straightforwardly seen to be both $P$-invariant. In particular, $Pe_0$ must take the form $Pe_0=\sum_{k\geq 0}a_ke_k$. For the same reason, $Pe_{-1}$ is in $H_-$, but it is also given by $Pe_{-1}=PU^*e_0=U^*Pe_o=\sum_{k\geq 0}a_ke_{k-1}$, which means $a_k=0$ for every $k>0$. In other words, $e_0$ must be an eigenvector of $P$. As such, we have either $Pe_0=0$ or $Pe_0=e_0$. In the first case $P=0$, whilst in the second $P=1$, because $Pf(U)e_0=f(U)Pe_0$ for every $f\in L^{\infty}(\mathbb{T})$.    
\end{proof}
\begin{corollary}
In the canonical representation $W^*(U)\cap \pi(O_2)''=\mathbb{C}1$. As a consequence, $C^*(U)'\cap \O_2=\mathbb{C}1$.
\end{corollary}

\begin{remark}
We have included an elementary proof of \ref{Intr} for the reader's convenience. However, note that the rank-one orthogonal projections onto $\mathbb{C}e_n$, with $n\in\mathbb{Z}$, all belong to $W^*(U)\vee\pi(\O_2)'$:
indeed, we have $U^nE_+(U^*)^n=E_{\overline{ {\rm span} }\left\{e_k:\; k\geq n\right\} }$ for any $n\in\mathbb{Z}$. In light of this, the above proposition is well worth comparing with the more far-reaching classical result that for any discrete group $\Gamma$ the von Neumann algebra
on $\ell_2(\Gamma)$ generated by $\lambda(\Gamma)$ and the multiplication operators $M_f$, with $f\in c_0(\Gamma)$,  is the whole $B(\ell_2(\Gamma))$. Obviously, our case corresponds to $\Gamma=\mathbb{Z}$.  
\end{remark}

At this stage, there is another step to take to improve our knowledge of the $C^*$-algebra generated by $U$. Indeed, we are yet to prove that  $C^*(U)$ is a maximal abelian subalgebra of $\Q_2$ as well. Since this task requires some technical work, we postpone the proof to the next section.

\section{Structure results}
\subsection{Two maximal abelian subalgebras}\label{SecMASA}
The goal of the present section is to tackle two structure problems for $\Q_2$, namely that both $C^*(U)$ and $\D_2$ are maximal abelian subalgebras.
We start with $C^*(U)$. The relative result is easily guessed, and yet its proof is unfortunately far from being straightforward, in that it needs some more refined tools such as  conditional expectations from $B(\H)$ onto a maximal subalgebra.  As is known, the proof of the existence of such conditional expectations can be traced back to the classic work of Kadison and Singer \cite{KS}, where 
the authors first described a general procedure to obtain them. Nowadays, the existence of conditional
expectations  of this sort is  preferably  seen as an immediate consequence of the injectivity of abelian von Neumann algebras. 
Even so, we do sketch the original procedure, not least because we make use of it in the proof of Theorem \ref{Faith}. This runs as follows.
Given $T\in B(\H)$, set $T^{|P}\doteq PTP+(I-P)T(I-P)$ for any projection $P$ in $W^*(U)$.  If $\{P_i\}$ is a generating sequence of projections of $W^*(U)$, then every cluster point of the sequence $\{T^{|P_1|P_2|\ldots |P_n}\}$ lies in $W^*(U)'=W^*(U)$, 
as explained in \cite{KS}. This enables us to define a conditional expectation from $B(\H)$ onto $W^*(U)$ associated with each ultrafilter $p\in\beta\mathbb{N}$ simply by taking the strong limit of the subnet corresponding to that ultrafilter.
\\
 
We next show as a key lemma to achieve our result that, for any  conditional expectation $E$ from $B(\H)$ onto $W^*(U)$, we have that
$E[S_\alpha S_\beta^*]$ is at worst a monomial in $U$. For the sake of clarity, our proof is in turn divided into a series of preliminary lemmata.

\begin{lemma}\label{Lemma-Exp-S_2}
With the notations set above, the equalities $E[S_2^k]=E[(S_2^*)^k]=0$ hold for $k\in\mathbb{N}$.
\end{lemma}
\begin{proof}
First note that the second equality is a straightforward consequence of the first thanks to the fact that $E[T^*]=E[T]^*$, which holds for every $T\in B(\H)$. The commutation rules $S_2^k U=U^{2^k} S_2^k$ give $E[S_2^k]U=U^{2^k}E[S_2^k]=E[S_2^k]U^{2^k}$. If we now set $f(U)\doteq E[S_2^k]$, we see that $f(z)(z-z^{2^k})=0$, hence $f(z)=0$ for every $z\in\mathbb{T}$. This says
$E[S_2^k]=0$
and we are done.
\end{proof}

\begin{lemma}
We have that $E[S_2^k(S_2^*)^k]=2^{-k}$ for every non-negative integer $k$.
\end{lemma}
\begin{proof}
To begin with, we observe that
$$
	E[S_\alpha S_\alpha^*]=U^hE[S_2^{|\alpha|}(S_2^*)^{|\alpha|}]U^{-h}=E[S_2^{|\alpha|}(S_2^*)^{|\alpha|}]
$$
for some positive integer $h$.
Since $1=\sum_{|\alpha|=k} S_\alpha S_\alpha^*$
we have that 
$$
1=\sum_{|\alpha|=k} E[S_\alpha S_\alpha^*]=2^k E[S_2^{|\alpha|}(S_2^*)^{|\alpha|}].
$$
This implies that $E[S_2^k(S_2^*)^k]=2^{-k}$.
\end{proof}

\begin{lemma}
We have that $E[S_2^k(S_2^*)^m]=0$ for $k, m\neq 0$, $k \neq m$.
\end{lemma}
\begin{proof}
Thanks to the last two lemmas, it suffices to show the statement for $k>m>0$.
By using the commutation rules $S_2^k U=U^{2^k} S_2^k$ we get $U^{2^k}E[S_2^k(S_2^*)^m]U^{-2^m}=E[S_2^k(S_2^*)^m]$. If we now set $f(U)\doteq E[S_2^k(S_2^*)^m]$, we 
obtain the functional equation $(z^{2^k-2^m}-1)f(z)=0$,  
which clearly implies $f(z)=0$ for every $z\in\mathbb{T}$. This shows that
$E[S_2^k(S_2^*)^m]=0$
and we are done.   
\end{proof}
 We now have all the necessary information to carry out our proof of $C^*(U)$ being a maximal abelian algebra.

\begin{theorem}\label{UMASA}
$C^*(U)$ is a maximal abelian $C^*$-subalgebra of $\Q_2$.
\end{theorem}
\begin{proof}
As $C^*(U)'\cap\Q_2=W^*(U)\cap\Q_2$, it is enough to prove that given $f\in L^\infty(\mathbb{T})$ with $f(U)$ in $\Q_2$, then $f$ is in fact a continuous function. 
Now if $f(U)$ belongs to $\Q_2$, it is also the norm limit of a sequence $\{x_k\}\subset\Q_2$ with each $x_k$ taking the form $\sum_{\alpha, \beta, h} c_{\alpha, \beta, h} S_\alpha S_\beta^*U^h$.
If $E: B(\H)\rightarrow W^*(U)$ is any of the conditional expectations considered above, we have $f(U)=E(f(U))=\lim_k E(x_k)$. But then each $E(x_k)$ is of the form $\sum_{\alpha, \beta, h} c_{\alpha, \beta, h}U^{h+k_{\alpha, \beta}}$, where
$U^{k_{\alpha, \beta}}$ is nothing but $E(S_\alpha S_\beta^*)$. Consequently, there exists a sequence of Laurent polynomials $p_k$ such that $\|f(U)- p_k(U)\|$ tends to zero, that is $f(U)\in C^*(U)$, as maintained.

\end{proof}
Among other things, it is interesting to note that the former proof yields a distinguished conditional expectation from $\Q_2$ onto $C^*(U)$, which is simply obtained by restricting any of the aforesaid conditional expectations to $\Q_2$.
Although there are conditional expectations onto $W^*(U)$ aplenty, as proved in \cite{KS}, the above computations also show that $E$ is in fact unique, a fact worth a statement of its own. 
\begin{theorem}
The conditional expectation $E:\Q_2\rightarrow C^*(U)$ is unique.  
\end{theorem} 
To complete the picture, we next show that $E$ is faithful. This is actually a straightforward consequence of a general well-known result due to Tomiyama \cite{Tomi}, whose proof in our setting is nevertheless included for the sake of completeness, being utterly independent of Tomiyama's work to boot.
\begin{proposition}\label{Faith}
The unique conditional expectation $E$ above is faithful.
\end{proposition}  
\begin{proof}
By uniqueness it is enough to make sure that the conditional expectation yielded by the Kadison-Singer procedure is faithful. Now if $T$ is an $\alpha$-coercive operator, i.e. $(Tx,x)\geq \alpha\|x\|^2 $ with $\alpha>0$, then $T^{|P}$ is $\alpha$-coercive as well regardless of the projection $P$. In particular, if $T\in B(\H)$ is a coercive operator, then $E[T]$ cannot zero, being by definition a weak limit of coercive positive operators all with the same constant as $T$. If now $T$ is any non-zero positive operator and
$\varepsilon>0$ is any real number with $\varepsilon<\|T\|$, the spectral theorem provides us with an orthogonal decomposition $\H=M_\varepsilon\oplus N_\varepsilon$ with $M_\varepsilon$ and $N_\epsilon$ both $T$-invariant and such that the restriction $T\upharpoonright_{N_\varepsilon}$ is $\varepsilon$-coercive. The remark we started our proof with allows to conclude that $E[T]$ is not zero either.
\end{proof}

We can now move on to $\D_2$. Again, the techniques we employ make a rather intensive use of conditional expectations. 
Before we start, it is worth mentioning that this result can be understood as a generalization of the well-known property of $\D_2$ being maximal 
in $\O_2$. We start  attacking the problem  with the following couple of lemmas, for which we first need to set some notation. We still denote by $E$ the unique faithful conditional expectation from $B(\H)$ onto $\ell_\infty(\mathbb{Z})$. As known, this is simply given by $(E[T]e_i, e_j)= (Te_i, e_j)\delta_{i,j}$.
\begin{lemma}
The following relations hold:
\begin{itemize}
\item $E[U^k]=\delta_{k,0}I$,
\item $E[S_\alpha S_\alpha^*]= S_\alpha S_\alpha^*$,
\item if $|\alpha |\neq |\beta|$, $E[S_\alpha S_\beta^*U^h]$ is either $0$ or  $E_i$, 
\item if $|\alpha |= |\beta|$, $E[S_\alpha S_\beta^*U^h]$ is either $0$ or  $S_\alpha S_\beta^*U^h$.
\end{itemize}
\end{lemma}
\begin{proof}
The first two equalities need no proof. For the third, without harming generality, we may suppose that $|\alpha |<|\beta|$
\begin{align*}
E[S_\alpha S_\beta^*U^h] & = E[U^{h(\alpha)}(S_2)^{|\alpha |} (S_2^*)^{|\beta |}U^{h-h(\beta)}]\\ 
 & = E[U^{h(\alpha)}(S_2)^{|\alpha |} (S_2^*)^{|\alpha |} U^{- h(\alpha)}U^{h(\alpha)} (S_2^*)^{|\beta |-|\alpha |} U^{h-h(\beta)}]\\
 & = U^{h(\alpha)}(S_2)^{|\alpha |} (S_2^*)^{|\alpha |} U^{- h(\alpha)} E[U^{h(\alpha)} (S_2^*)^{|\beta |-|\alpha |} U^{h-h(\beta)}].
\end{align*}
where $h(\alpha)$ and $h(\beta)$ are positive integers.
Accordingly, we are led to compute $E[U^{h} (S_2)^{k} U^{l}]$, where $h\in \mathbb{Z}$.  Now the condition $U^{h} (S_2)^{k} U^{l} e_i=e_i$
implies that $i=2^k(i+l)+h$, and because the former equation has a unique solution, we get the thesis. Finally, for the fourth we have that
\begin{align*}
E[S_\alpha S_\beta^*U^h] & = E[U^{h(\alpha)}(S_2)^{|\alpha |} (S_2^*)^{|\beta |}U^{h-h(\beta)}]\\ 
 & = E[U^{h(\alpha)}(S_2)^{|\alpha |} (S_2^*)^{|\alpha |} U^{-h(\alpha)} U^{h-h(\beta)+h(\alpha)}]\\
 & = U^{h(\alpha)}(S_2)^{|\alpha |} (S_2^*)^{|\alpha |} U^{-h(\alpha)} E[U^{h-h(\beta)+h(\alpha)}]\\
 & = \delta_{h-h(\beta)+h(\alpha),0} U^{h(\alpha)}(S_2)^{|\alpha |} (S_2^*)^{|\alpha |} U^{-h(\alpha)}\\
 & = \delta_{h-h(\beta)+h(\alpha),0} U^{h(\alpha)}(S_2)^{|\alpha |} (S_2^*)^{|\alpha |} U^{h-h(\beta)}\\
 & =  \delta_{h-h(\beta)+h(\alpha),0} S_\alpha S_\beta^*U^h.
 \end{align*}
\end{proof}
In order to make our proof work, we also need to take into account the conditional expectation $\Theta$ from $\Q_2$ onto $\D_2$ described in \cite{LarsenLi}. We recall that this is uniquely determined by 
$\Theta ((S_2^*)^i U^{-l}fU^{l'} S_2^{i'}) \doteq  \delta_{i,i'} \delta_{l,l'} (S_2^*)^i U^{-l}fU^{l} S_2^{i}$, where $f$ is in $\F_2$. Moreover, it is there shown to be faithful too.
\begin{lemma}
If $|\alpha |= |\beta|$, $\Theta [S_\alpha S_\beta^*U^h]$ is either $0$ or $S_\alpha S_\beta^*U^h$. In particular, $\Theta$ and $E$ coincide on monomials $S_\alpha S_\beta^* U^h$ with $|\alpha|=|\beta|$. 
\end{lemma}
\begin{proof}
By direct computation. Indeed, we have
\begin{align*} 
\Theta [S_\alpha S_\beta^*U^h] &=\Theta (U^{h(\alpha)}(S_2)^{|\alpha |} (S_2^*)^{|\alpha |} U^{h-h(\beta)})\\
& =\delta_{h(\alpha ), -h+h(\beta)}U^{h(\alpha)}(S_2)^{|\alpha |} (S_2^*)^{|\alpha |} U^{h-h(\beta)}
\end{align*}
\end{proof}
\begin{theorem}\label{DMASA}
The diagonal subalgebra $\D_2\subset\Q_2$ is a maximal abelian subalgebra.  
\end{theorem}
\begin{proof}
As usual, all we have to do is make sure that the relative commutant $\D_2'\cap\Q_2=\ell_\infty({\mathbb{Z}}) \cap\Q_2$ reduces to $\D_2$. Let $x\in \ell_\infty(\mathbb{Z})\cap\Q_2$, then there exists a sequence 
$\{x_k\}$ converging normwise to $x$ with each of the $x_k$ of the form $\sum c_{\alpha, \beta, h} S_\alpha S_\beta^* U^h$. As above, $x=E(x)=\lim_k E(x_k)$. Thanks to the former lemmata, 
we can rewrite $E(x_k)$ as $d_k+ f_k$, where $d_k\in \D_2$ and $f_k$ are all diagonal finite-rank operators. Now, being $d_k=\Theta(x_k)$, we see that $d_k$ must converge to some $d\in\D_2$. 
But then $f_k$ converge normwise to a diagonal compact operator, say $k$, which means $k=x-d$ is in $\Q_2$, hence $k=0$, 
being $\K(\H) \cap \Q_2 = \{0\}$,
and $x=d\in\D_2$.   
\end{proof}
The former result readily leads to $\D_2$ being a Cartan subalgebra of $\Q_2$.
An anonymous referee pointed out an alternative approach to show this by using the works of N. Larsen and X. Li, \cite{LarsenLi}, and J. Renault, \cite{Renault}. We believe this approach is very elegant and we give a brief account of it. In the first place, one should see the $2$-adic ring $C^*$-algebra $\Q_2$ as a reduced groupoid  $C^*$-algebra (see \cite[Section 5]{LarsenLi}). 
Indeed, if $G=({\mathbb Z}[\frac{1}{2}] \rtimes 2^{\mathbb Z}) \ltimes {\mathbb Q}_2$ is the transformation groupoid, where ${\mathbb Z}[\frac{1}{2}]$ and $2^{\mathbb Z}$ act on ${\mathbb Q}_2$ by addition and multiplication, respectively, and $F:=G|_{{\mathbb Z}_2}$, then $C^*_r(F) \simeq \Q_2$.
Now, the claim follows at once by using \cite[Proposition 3.1]{Renault} 
after showing that the set of isotropy points of $\mathbb{Z}_2$ is countable, and thus that the set of points with trivial isotropy is dense (that is, the corresponding action is topologically free \cite[Sec. 6.1]{Renault}).
The canonical abelian subalgebra of a crossed product with respect to a group is considered as the most typical example of MASA, provided that the action is topologically free. In the present case, although $\Q_2$ is only a crossed product by a semigroup, the situation is almost the same because it is just a corner of a crossed product by a group. One advantage of this description of $\Q_2$ as a crossed product 
makes more natural to recognize $\D_2$ as the canonical Cartan subalgebra (and the canonical expectation).

\smallskip

As for $\O_2$, later on we will prove that there is no conditional expectation from $\Q_2$ onto $\O_2$.

\subsection{Irreducible subalgebras}\label{SecIrr}
In order to take a step further towards the study of $\Q_2$, especially as far as the properties of the inclusion $\O_2\subset \Q_2$ are concerned,  it is worthwhile to recall a useful result proved by Larsen and Li in their aforementioned paper \cite{LarsenLi}. It says that a representation $\rho$ of $\O_2$ extends to a representation of $\Q_2$ if and only if the unitary part of the Wold decomposition of $\rho(S_1)$ and $\rho(S_2)$ are
unitarily equivalent. Accordingly, once the unitary parts of the Wold decompositions have been proved to be unitarily equivalent,  the isometries  are unitarily equivalent too because of the relation $US_1 = S_2 U$.  
As remarked by the authors themselves, this allows us to think of $\Q_2$ as a sort of symmetrized version of $\O_2$.
Notably, the result applies to those representations $\pi$ of $\O_2$ in which both $\pi(S_1)$ and $\pi(S_2)$ are pure. Moreover, in such cases the sought extension is unique, as pointed out in \cite[Remark 4.2]{LarsenLi}. 
For the reader's convenience, though, we do single out this as a separate statement.
\begin{proposition}\label{prop-uniq-ext}
Let $\pi$ be a representation of $\O_2$ on the Hilbert space $\H_\pi$ such that $\pi(S_1)$ and $\pi(S_2)$ are both pure. Then there exists a unique unitary $\widetilde U\in B(\H_\pi)$ such that $\pi(\S_2) \widetilde U=\widetilde U^2 \pi(S_2)$ and $\pi(S_1)=\tilde U\pi(S_2)$.
\end{proposition} 
The recalled theorem by Larsen and Li would actually be enough to prove that the Cuntz algebra $\O_2$ is irreducible in $\Q_2$, as might be expected. 
Even so, this can also be derived from the much stronger result that even the UHF subalgebra $\F_2$ has trivial relative commutant, which is shown below.
\begin{theorem}
The UHF subalgebra $\F_2\subset\Q_2$ is irreducible, i.e. $\F_2' \cap\Q_2=\mathbb{C}1$.
\end{theorem}
\begin{proof}
As $\D_2$ is a subalgebra of $\F_2$,  we have $\F_2'\cap\Q_2\subset\D_2'\cap\Q_2=\D_2$, where the last equality depends upon $\D_2$ being a maximal abelian subalgebra of $\Q_2$. 
Therefore, we find $\F_2'\cap\Q_2=\F_2'\cap \D_2\subset \F_2'\cap\F_2=Z(\F_2)=\mathbb{C}1$.
\end{proof}
While immediately derived from the above theorem, the following couple of corollaries do deserve to be highlighted.  
\begin{corollary}
The relative commutant of $\O_2$ in $\Q_2$ is trivial, i.e.  $\O_2'\cap\Q_2=\mathbb{C}1$.
\end{corollary}
\begin{corollary}
The relative commutant of the gauge-invariant subalgebra $\Q_2^{\mathbb T}$ is trivial, i.e. $(\Q_2^{\mathbb T})' \cap \Q_2=\mathbb{C}1$.
\end{corollary}

Interestingly, the irreducibility of $\F_2$  also applies to the canonical shift, which  turns out to enjoy the so-called shift property, i.e. $\bigcap_k \widetilde\varphi^k(\Q_2)=\mathbb{C}1$, whence its name. 
This important property should have first been singled out by R. T. Powers, who called it strong ergodicity, but we do not have a precise reference for the reader. The canonical shifts of $\O_n$ are of course known to be strongly ergodic, see \cite{Laca} for a full coverage of the topic.

\begin{theorem}\label{Shift}
The canonical endomorphism $\widetilde\varphi$ of $\Q_2$ is a shift, i.e. $\bigcap_k \widetilde\varphi^k(\Q_2)=\mathbb{C}1$.
\end{theorem}
\begin{proof}
Since $\widetilde\varphi^k(x)=\sum_{\gamma: |\gamma|=k} S_\gamma x S_\gamma^*$, the equality $\widetilde\varphi^k(x) S_\alpha S_\beta^*= S_\alpha S_\beta^*\widetilde\varphi^k(x)$, $x\in\Q_2$, is
 straightforwardly checked
to hold true for every pair of multi-indices $\alpha$ and $\beta$ with $|\alpha|=|\beta|=k$. This says that  $\widetilde\varphi^k(\Q_2)$ is contained in $(\F_2^k)'\cap\Q_2$ for every $k$, where $\F_2^k\doteq
\textrm{span}\{S_\mu S_\nu^*: |\mu|=|\nu|\}$. But then we have the chain of inclusions
 $\bigcap_k\widetilde\varphi^k(\Q_2)\subset\bigcap_k (\F_2^k)'\cap\Q_2\subset \F_2'\cap\Q_2=\mathbb{C}1$.
\end{proof}

\begin{remark}
The former theorem says in particular that $\widetilde\varphi$ is not surjective in a rather strong sense. We can be more precise  by observing that $U$ is not in $\widetilde\varphi(\Q_2)$. 
Indeed, by maximality of $C^*(U)$, any inverse image of $U$ should lie in $C^*(U)$, but the restriction of $\widetilde\varphi$ to $C^*(U)$ does not yield a homeomorphism of $\mathbb{T}$. 
\end{remark}

We can now go back to the announced result that conditional expectations onto the Cuntz algebra $\O_2$ do not exist.

\begin{theorem}\label{Nocondexp}
There is no unital conditional expectation from $\Q_2$ onto $\O_2$.
\end{theorem}
\begin{proof}
Suppose that such a conditional expectation does exist. We want to show that this leads to $E(U)$ being $U$, which is obviously absurd. We shall work in any representation in which $S_1$ and $S_2$ are both pure, for instance the one described in \cite{Latremoliere}.
If we compute $E$ on the operator $US_1^nS_2S_2^*(S_1^*)^n$ by using the commutation rule $S_2^nU=US_1^n$ we easily get to the equality
$$E[U]S_1^nS_2S_2^*(S_1^*)^n=S_2^nS_1S_2^*(S_1^*)^n$$
But on the other hand we also have that $U S_1^nS_2S_2^*(S_1^*)^n= S_2^nS_1S_2^*(S_1^*)^n$.
Accordingly, $E(U)$ and $U$ must coincide on the direct sum of the subspaces $M_n\doteq S_1^nS_2S_2^*(S_1^*)^n\H_\pi$, which can be easily seen to be the whole $\H_\pi$. 
\end{proof}

\subsection{The relative commutant of the generating isometry}\label{Secrelcomm}

This section is entirely devoted to proving that  $C^*(S_2)'\cap \Q_2$ is trivial. We first observe that this is the same as proving that
$C^*(S_1)'\cap\Q_2$ is trivial, merely because $\textrm{ad}(U^*)(C^*(S_2))=C^*(S_1)$. For this, we still need some preliminary definitions and results. \\

Given any $k\in\mathbb{N}$, we set $\B_2^k\doteq {\rm span}\{S_\alpha S_\beta^* U^h \ | \ |\alpha|=|\beta|=k, h\in \mathbb{Z}\}$. For every $k$ we have the inclusion $\B_2^k \subset \B_2^{k+1}$, which
readily follows from the Cuntz relation $1=S_1S_1^*+S_2S_2^*$.
\begin{lemma}
Let $x\in \B_2^k$, then
\begin{enumerate}
\item $(S_1^*)^k xS_1^k\in C^*(U)$;
\item the sequence $\{(S_1^*)^mxS_1^m\}$ stabilizes to a scalar $c_x\in \mathbb{C}$.
\end{enumerate} 
\end{lemma}
\begin{proof}
Without loss of generality, suppose that $x=S_\alpha S_\beta^* U^h$.
We have that
\begin{align*}
(S_1^*)^k x S_1^k & = \delta_{\alpha, \underline{1}} S_\beta^* U^h S_1^k .
\end{align*}
If $h>0$, then by using the relation $S_2U=US_1$, we see that the r.h.s. of the last expression is given by
$$
\delta_{\alpha, \underline{1}} S_\beta^* S_\gamma U^{l(h)}=\delta_{\alpha, \underline{1}} \delta_{\beta, \gamma} U^{l(h)} $$
where $\gamma$ is a multi-index of length $k$.
If $h=0$, we obtain
$$
\delta_{\alpha, \underline{1}} S_\beta^* S_1^k=\delta_{\alpha, \underline{1}} \delta_{\beta, \underline{1}} .
$$
When $h<0$, by using the relations $U^*S_1=S_2$ and $U^*S_2=S_1U^*$ we obtain
$$
\delta_{\alpha, \underline{1}} S_\beta^* U^{h} S_1^k= \delta_{\alpha, \underline{1}} U^{l(h)}  S_\gamma^* S_{\underline{1}} =\delta_{\alpha, \underline{1}} \delta_{\gamma,\underline{1}} U^{l(h)} 
$$
where $\gamma$ is a multi-index of length $k$.  We observe that in these cases we always have $|l(h)|\leq |h|$, and this fact will be important in the sequel. For the second part of the thesis, we may suppose that $m > k+|h|+1$. The needed computations can be made faster in the canonical representation (for brevity we write $l$ instead of $l(h)$):
\begin{align*}
(S_1^*)^mU^l S_1^m e_j & = (S_1^*)^mU^l e_{2^mj+2^m-1} = (S_1^*)^m e_{2^mj+2^m-1+l}.
\end{align*}
The expression above is non-zero if and only if $2^mj+2^m-1+l=2^m i+2^m-1$ for some $i$, that is to say $l=2^m(i-j)$. But $m>k+h+1$ and 
$|l|\leq |h|$, therefore we finally get $i=j$ and $l=0$.

\end{proof}

\begin{proposition}\label{limfixpoint}
Let $x\in \Q_2^\mathbb{T}=\overline{{\rm span}}\{S_\alpha S_\beta^* U^h | |\alpha|=|\beta|, h\in \mathbb{Z}\}$. Then 
$$
\lim_h (S_1^*)^h x S_1^h\in \mathbb{C}
$$
\end{proposition}
\begin{proof}
By hypothesis there exists a sequence $x_k\in \B_2^{f(k)}$ that tends normwise to $x$. Choose a pair of natural numbers  $i$ and $j$. For any $k\in \mathbb{N}$ sufficiently larger than $f(i)$ and $f(j)$, by the former lemma we have that $(S_1^*)^k x_iS_1^k=:c_i, (S_1^*)^k x_jS_1^k=:c_j \in \mathbb{C}$. The sequence $c_i$ is convergent since
$$
|c_i-c_j|=\|(S_1^*)^k x_iS_1^k-(S_1^*)^k x_jS_1^k \|\leq \|x_i-x_j\| .
$$
We denote by $c$ the limit. Now the sequence $(S_1^*)^h x S_1^h$ tends to $c$ as well.
\end{proof}
For any non-negative integer $i$ we now define the linear maps $F_i: \Q_2\to \Q_2^{\mathbb{T}}$ given by
\begin{align*}
F_i(x)& \doteq \int_\mathbb{T} \widetilde{\alpha}_z [x (S_1^*)^i] dz \\
F_{-i}(x)& \doteq \int_\mathbb{T} \widetilde{\alpha}_z [S_1^i x] dz. 
\end{align*}
We observe that
\begin{align} 
F_i(x)& = F_i(x) S_1^i(S_1^*)^i \label{formula-coeff}\\
F_{-i}(x)& = S_1^i(S_1^*)^i F_{-i}(x). \label{formula-coeff-2}
\end{align}
Before proving the main result of the section,  we also need to recall the following lemma, whose proof can be adapted verbatim from the original \cite[Proposition 1.10]{Cuntz1}, where it is proved for the Cuntz algebras instead.
\begin{proposition} \label{Cuntzlemma}
Let $x\in \Q_2$ be such that $F_i(x)=0$ for all $i\in \mathbb{Z}$. Then $x=0$.
\end{proposition}
Now we have all the tools  for completing our proof.
\begin{theorem}\label{Relcomm}
Let $w\in \U(\Q_2)$ such that $wS_1w^*=S_1$. Then $w\in \mathbb{T}1$.
\end{theorem}
\begin{proof}
First of all we observe that we also have $wS_1^*w^*=S_1^*$. We have that $S_1^* F_i(w)S_1=F_i(w)$. Indeed, for $i \geq 0$,
\begin{align*}
S_1^*F_i(w)S_1& =S_1^*\left( \int_\mathbb{T} \widetilde{\alpha}_z [w (S_1^*)^i] dz \right) S_1 = \int_\mathbb{T} S_1^*\widetilde{\alpha}_z [w (S_1^*)^i] S_1 dz \\
& = \int_\mathbb{T} \widetilde{\alpha}_z [S_1^*w (S_1^*)^i S_1] dz  = \int_\mathbb{T} \widetilde{\alpha}_z [w (S_1^*)^i] dz=F_i(w) \\
S_1^* F_{-i}(w) S_1& =S_1^* \left( \int_\mathbb{T} \widetilde{\alpha}_z [S_1^i w] dz \right) S_1 =\int_\mathbb{T} S_1^*\widetilde{\alpha}_z [S_1^i w] S_1dz \\ 
& =\int_\mathbb{T} \widetilde{\alpha}_z [S_1^*S_1^i wS_1] dz  =\int_\mathbb{T} \widetilde{\alpha}_z [S_1^i w] dz=F_{-i}(w) \\
\end{align*}
By Proposition \ref{limfixpoint} we obtain that for each $i \in {\mathbb Z}$ one has
$$
\lim (S_1^*)^k F_i(w)S_1^k=F_i(w)\in \mathbb{C}.
$$
Equation \eqref{formula-coeff}-\eqref{formula-coeff-2} together imply that for $i\neq 0$ we have $F_i(w)=0$. Now Proposition \ref{Cuntzlemma} applied to $w-F_0(w)$ gives the claim. 
\end{proof}
Exactly as for the Cuntz algebras, we can also state a slight generalization of the former result, which says an inner automorphism of $\Q_2$ cannot send $S_i$ to a scalar multiple of it.

\begin{proposition}\label{zS}
Let $\phi\in\Aut(\Q_2)$ be such that $\phi(S_i)=z S_i$ for some $z\in \mathbb{T}\setminus\{1\}$ and  
$i=1$ or $i=2$. Then $\phi$ is an outer automorphism.
\end{proposition}
The proof is straightforwardly obtained by recasting Theorem \ref{Relcomm} in terms of a unitary $w$ such that $wS_1w^*=zS_1$. This is done for $\O_n$ in \cite{Matsumoto}, from which some of the techniques deployed in this section have actually  been taken.

\section{Extending endomorphisms of the Cuntz algebra}\label{SecExt}

For what follows it may be convenient to recall that associated with any unitary $V \in \U(\O_2)$ there is an endomorphism $\lambda_V$ of $\O_2$, defined by $\lambda_V(S_i)=VS_i, i=1,2$. 
Notably, the other way round is also true, i.e. any endomorphisms $\lambda\in\O_2$ comes from a unique $V$.  The bijective correspondence thus obtained is often named
after Cuntz and Takesaki. To take just one important example, the flip-flop $\lambda_f$ is nothing but the
automorphism corresponding to $f\doteq S_1 S_2^* + S_2 S_1^*$.

\subsection{Uniqueness of the extensions}
This section is mostly concerned with the problem of extending endomorphisms of $\O_2$ to endomorphisms of $\Q_2$. More precisely, we spot necessary and sufficient conditions for an extension to exist. 
Before entering into the details, some comments are in order. Indeed, the equations we get are hardly ever easy to verify save for the endomorphisms we already know of to extend. Notwithstanding their intrinsic difficulty, they do provide
general information when applied to $\textrm{id}_{\O_2}$. We can now go on with our discussion. To begin with,  if $V$ is a unitary of $\O_2$ such that the associated endomorphism $\lambda_V$ extends to an endomorphism $\tilde\lambda$ of $\Q_2$, then one must have $VUS_2 = VS_1 = \lambda_V(S_1) = \tilde\lambda(S_1)
= \tilde\lambda(U) \lambda_V(S_2) = \tilde\lambda(U) V S_2$. Therefore,
$\tilde\lambda(U)V S_2 = VU S_2$ and thus $U^* V^* \tilde\lambda(U)V S_2 = S_2$. Thus, setting $W = U^* V^* \tilde\lambda(U)V $,
it holds $W S_2 = S_2$ and $\tilde\lambda(U) = V U W V^*$.
We now examine whether such extensions exist. 
We always have 
$$V S_2 S_2^* V^* + V U W V^* V S_2 S_2^* V^* V W^* U^* V^* =
 V (S_2 S_2^* +   U S_2 S_2^* U^*) V^*= VV^* = 1 $$ 
and we must have
$$V S_2 \tilde\lambda(U) = \tilde\lambda(U^2) V S_2$$
or, equivalently,
$$VS_2 VU W V^* = (VUWV^*)^2 V S_2 = V U W V^* V U W V^* V S_2 = V U W U W S_2 = 
V U W U S_2 \ . $$
We have thus shown the following result.
\begin{proposition}
Let $V \in \U(\O_2)$ and let $\lambda_V\in\End(\O_2)$ be the associated endomorphism. Then $\lambda_V$ extends to an endomorphism of $\Q_2$  if and only if there exists a unitary $W \in {\mathcal Q}_2$ such that $W S_2 = S_2$
and $S_2 VU W V^* = U W U S_2$.
For any such $W$, we have an extension $\tilde\lambda = \tilde\lambda_{V,W}$ with $\tilde\lambda(U) = V U W V^*$.
\end{proposition}
As shown later, the $W$ defined above is uniquely determined in all the cases we have examined. 
Furthermore, the endomorphism $\widetilde\lambda$ is necessarily injective because of $\Q_2$ being simple. Moreover, if $\lambda_V$ is an automorphism of $\O_2$, then
$\widetilde\lambda$ is surjective if and only if the associated $W$ is contained in $\widetilde\lambda(\Q_2)$. Moreover, for the extensions built above the following composition rule holds:
$$\widetilde\lambda_{V,W} \circ \widetilde\lambda_{V',W'} = \widetilde\lambda_{\lambda_V(V')V, W V^* \widetilde\lambda_{V,W}(W') V}  $$

As an example, if $\tilde\varphi$ is the canonical shift introduced in Sect. \ref{SecPrep}, we have $$\widetilde\varphi = \widetilde\lambda_{\theta, U^* \theta U^2 \theta}$$ where $\theta = \sum_{i,j=1}^2 S_i S_j S_i^* S_j^* \in \U(\F_2^2)$ is the self-adjoint unitary flip.\\

It is interesting to note that the extensions of the gauge automorphisms we have considered all work with $W=1$. This is not a case. In fact, the converse also holds true.
\begin{proposition}
Let $V\in\U(\O_2)$. If the associated endomorphism $\lambda_V\in\End(\O_2)$ extends to $\tilde\lambda_{V, 1}$, that is to say the choice $W=1$ does yield an extension,   then $V=z1$, for some $z\in\mathbb{T}$.  
\end{proposition}
\begin{proof}
If we put $W=1$ in the equality $S_2VUWV^*=UWUS_2$, we get $S_2VUV^*=U^2S_2$. But $U^2S_2=S_2U$, and so we must have $S_2VUV^*=S_2U$. Hence $VUV^*=U$, i.e. $V$ commutes with $U$. Since $V$ is a unitary, we also have $V\in C^*(U)'\cap\O_2$, which concludes the proof.
\end{proof}

Extensions of the identity map of $\O_2$, which obviously correspond to $V=1$, may be looked at more closely. If we define $W\doteq U^* \tilde\lambda(U)$, we find that $W$ is a unitary in $\Q_2$ such that $\tilde\lambda(U) = UW$, $W S_2 = S_2$ and $W S_1 = S_1 W$. Indeed, 
$S_2 S_2^* + UW S_2 S_2^* W^* U^* = S_2 S_2^* + U S_2 S_2^* U^* = 1$ and
$S_2 UW = (UW)^2 S_2$, so that $U^2 S_2 W = U W U W S_2$ and thus $U S_2 W = W U S_2$.
Hence, $S_1 W = W S_1$, as stated. Obviously, the trivial choice $W=1$ corresponds to the trivial extension. 
\begin{proposition}
If $W\in\U(\Q_2)$ is such that $WS_2=S_2$ and $WS_1=S_1W$, then $W=1$.
\end{proposition}
\begin{proof}
This is in fact a straightforward application of the fact that $C^*(S_1)'\cap\Q_2=\mathbb{C}1$. However, we can also give an alternative if longer proof that only depends on the theorem of Larsen and Li we have quoted. The computations are easily made in the irreducible representation of $\O_2$  produced in \cite{Latremoliere}. This acts on the Hilbert space 
$\H=L^2([-1,1])$ through the pure isometries $S_1, S_2\in B(\H)$ given by the formulas:
\begin{displaymath}
(S_1 f)(t) = \left\{ \begin{array}{ll}
0 & \textrm{for $-1\leq t\leq 0$}\\
\sqrt{2}f(2t-1) & \textrm{for $0< t\leq 1$}\\
\end{array} \right.
\end{displaymath}

\begin{displaymath}
(S_2 f)(t) = \left\{ \begin{array}{ll}
\sqrt{2}f(2t+1) & \textrm{for $-1\leq t\leq 0$}\\
0 & \textrm{for $0<t\leq 1$}\\
\end{array} \right.
\end{displaymath}
Note that $(S_1^*f)(t)=\frac{\sqrt{2}}{2}f\big(\frac{t+1}{2}\big)$ and $(S_2^*f)(t)=\frac{\sqrt{2}}{2}f\big(\frac{t-1}{2}\big)$ for every $f\in \H$. The unitary operator $U\in U(\H)$ given by $(Uf)(t)=f(-t)$ for every $t\in [-1,1]$ is an intertwiner between $S_1$ and $S_2$, namely $US_2=S_1U$.  By virtue of the result of Larsen and Li we mentioned above, we can then regard this representation as a representation of $\Q_2$ as well, which allows us to think of $\Q_2$ as a subalgebra of $B(\H)$. In order to prove the proposition, we will actually show even more: any unitary $W\in B(\H)$ such that $WS_2=S_2$ and $WS_1=S_1W$ must be the identity operator $I$ on $\H$. To accomplish this task, we define a sequence of orthogonal projections  given by $P_n\doteq (S_1)^nS_2S_2^*(S_1^*)^n$  for each $n\in\mathbb{N}$. It is  straightforwardly checked that $WP_n=P_n$ for every $n$ and $P_nP_m=P_mP_n=0$ for every $m\neq n$. Therefore $Q_n\doteq\sum_{k=0}^n P_k$ is still an orthogonal projection such that $WQ_n=Q_n$. Accordingly, the conclusion will be easily achieved once we have proved that $Q_n$ converges to $I$ in the strong operator topology. As easily recognized, we have the following explicit formulas for $S_1^n$ and $(S_1^*)^n$:

\begin{displaymath}
(S_1^n f)(t) = \left\{ \begin{array}{ll}
0 & \textrm{for $-1\leq t\leq 1-\frac{1}{2^{n-1}}$}\\
\sqrt{2}^nf(2^nt-\sum_{i=0}^k 2^i) & \textrm{for $1-\frac{1}{2^{n-1}}< t\leq 1$}\\
\end{array} \right.
\end{displaymath}
and
\begin{displaymath}
((S_1^*)^n f)(t) =\left(\frac{\sqrt{2}}{2}\right)^n f\left(\frac{t+2^n-1}{2^n}\right) 
\end{displaymath}
We can use them to see that $P_n$ is the projection corresponding to the multiplication operator by $\chi_{\big[1-\frac{1}{2^{n-1}}, 1-\frac{1}{2^n}\big]}$, i.e. the characteristic function of the interval $\big[1-\frac{1}{2^{n-1}}, 1-\frac{1}{2^n}\big]$. As a consequence,
$Q_n$ is nothing but the projection associated with $\chi_{\big[-1, 1-\frac{1}{2^n}\big]}$. Hence $Q_n\rightarrow 1$ in the strong operator topology, which was to be proved.
\end{proof}
\begin{remark}
The representation of $\Q_2$ recalled in the proof of the above result is not equivalent to the canonical representation, merely because its restriction to $\O_2$ is still irreducible, as proved in \cite{Latremoliere}, whereas the restriction of the canonical representation to $\O_2$ is not, as we remarked.
\end{remark}

We are at last in a position to prove the following result that says that a non-trivial endomorphism of $\Q_2$ cannot fix $\O_2$ pointwise.
\begin{theorem}\label{Rigidity}
If $\Lambda\in\End(\Q_2)$ is such that $\Lambda\upharpoonright_{\O_2}={\rm id}_{\O_2}$, then $\Lambda={\rm id}_{\Q_2}$.
\end{theorem}
\begin{proof}
A straightforward application of the former proposition.
\end{proof} 
\begin{remark}
Actually, the theorem just obtained strengthens the information that the relative commutant $\O_2'\cap\Q_2$ is trivial. For, if $u\in O_2'\cap\Q_2$, then $\textrm{ad}(u)$ is an automorphism fixing $\O_2$ pointwise. As such, $\textrm{ad}(u)$ is trivial, hence $u$ is a central element. Since $\Q_2$ is simple, $u$ must be a multiple of the identity, i.e. $\O_2'\cap\Q_2=\mathbb{C}1$.
\end{remark}
As a simple corollary, we can also get the following property of the inclusion $\O_2\subset\Q_2$. 
\begin{corollary}
If $\Lambda_1\in\Aut(\Q_2)$ and $\Lambda_2\in\End(\Q_2)$ are such that $\Lambda_1\upharpoonright_{\O_2}=\Lambda_2\upharpoonright_{\O_2}$, then $\Lambda_1=\Lambda_2$. In particular, $\Lambda_2$ is an automorphism as well.
\end{corollary}
\begin{proof}
Just apply the above theorem to the endomorphism $\Lambda_1^{-1}\circ\Lambda_2$, which restricts trivially to $\O_2$.
\end{proof}
In particular, the extensions of both the flip-flop and the gauge automorphisms are unique.\\

Of course there are automorphisms of ${\mathcal Q}_2$ that do not leave $\O_2$ globally invariant. The most elementary example we can come up with is probably $\textrm{ad}(U)$. Indeed, ${\rm ad}(U)S_1 = U S_1 U^* = S_2 = U S_1, \quad 
{\rm ad}(U)S_2 = U S_2 U^* = S_1 U^* = U^* S_2 \ $
Hence, ${\rm ad}(U)(\O_2)$ is not contained in $\O_2$, because $S_1 U^*$ is not in $\O_2$. 
Even more can be said. Indeed, ${\rm ad}(U)(\F_2)$ is not contained in $\O_2$ either. This is seen as easily as before, since
for instance ${\rm ad}(U)(S_1S_2^*)=US_1S_2^*U^*$ does not belong to $O_2$ although $S_1S_2^*$ belongs to $\F_2$. Given that $US_1S_2^*U^*=S_2US_2^*U^*=S_2US_1^*UU^*=S_2US_1^*$, if $US_1S_2^*U^*$ were in $\O_2$, then $U=S_2^*S_2US_1^*S_1$ would in turn be in $\O_2$, which is not. 
Even so, ${\rm ad}(U)$ does leave the diagonal subalgebra $\D_2$ globally invariant. This can be shown by means of easy computations involving the projections of $\D_2^k\doteq {\rm span}\{S_\alpha S_\alpha^*\; \; {\rm s. t. } \; |\alpha|=k\}$ for every $k\in\mathbb{N}$.\\

We would like to end this section by remarking that for each $\Lambda\in\End(\Q_2)$ there still exists a unique $u_\Lambda\in\U(\Q_2)$ such that $\Lambda(S_2)=u_\Lambda S_2$ and $\Lambda(S_1)=u_\Lambda S_1$, which is simply given by
$u_\Lambda=\Lambda(S_1)S_1^*+\Lambda(S_2)S_2^*$. Furthermore, $\Lambda$ leaves $\O_2$ globally invariant if and only if $u_\Lambda\in\O_2$. This allows us to regard the map $\End(\Q_2)\ni\Lambda\rightarrow u_\Lambda\in\U(\Q_2)$ as a generalization of the well-known Cuntz-Takesaki correspondence. Nevertheless, this map is decidedly less useful for $\Q_2$ than it is for $\O_2$, not least because it is not surjective. In other words, there exist unitaries $u$ in $\U(\Q_2)$ such that 
the correspondence $S_1\rightarrow uS_1, \; S_2\rightarrow uS_2$ do not extend to any endomorphism of $\Q_2$. Examples of such $u$ are even found in $\U(\O_2)$, as we are going to see in the next section, where we shall give a complete description of the extensible
Bogoljubov automorphisms.  For the time being we observe that if a unitary $u\in\U(\Q_2)$ does give rise to an endomorphism $\Lambda_u$, the equation $uUS_2=\Lambda_u(U)uS_2$ must be satisfied. This says that $\Lambda_u(U)=uUWu^*$ for some $W\in\U(\Q_2)$ such that
$WS_2=S_2$ and $S_2uUWu^*=UWUS_2$. By the same computations as at the beginning of the section, the converse is also seen to be true. Hence we obtain a complete if hitherto unmanageable description of $\End(\Q_2)$. At any rate, our guess is that
the above equations are hardly ever verified unless $u$ is of a very special form, such as $u=v\widetilde\varphi(v^*)$ for any $v\in\U(\Q_2)$, corresponding to inner automorphisms, $u=z 1$, corresponding to the gauge automorphisms $\tilde\alpha_z$, or $u=S_2S_2^*U^*+US_2S_2^*$, corresponding to the flip-flop. In fact, this prediction is partly supported by the result in the negative obtained in the next section. Moreover, it is still not clear at all whether
the map $\End(\Q_2)\ni \Lambda\rightarrow u_\Lambda$ is injective, although its restriction to $\Aut(\Q_2)$ certainly is.

\subsection{Extensible Bogoljubov automorphisms}
We have seen some remarkable classes of automorphisms of $\O_2$ that  extend to $\Q_2$. However, there is no a priori reason to expect every endomorphism of $\O_2$ to extend automatically to an endomorphism of $\Q_2$. In fact, we next give rather elementary examples of automorphisms of $\O_2$ that do not extend. Indeed, if we denote by $\eta_{\alpha,\beta}$ the automorphism of $\O_2$ defined by $\eta_{\alpha,\beta}(S_1)=\alpha S_1$ and $\eta_{\alpha,\beta}(S_2)=\beta S_2$ for any given $\alpha,\beta\in\mathbb{T}$, we have the following proposition.
\begin{proposition}
The automorphisms $\eta_{\alpha,\beta}\in\Aut(\O_2)$  defined above extend to endomorphisms of $\Q_2$ if and only if $\alpha= \beta$.
\end{proposition}
\begin{proof}
Since $S_1$ and $S_2$ are unitarily equivalent in $\Q_2$, their images  $\alpha S_1$ and $\beta S_2$ would be unitarily equivalent as well if an extension of $\eta_{\alpha,\beta}$ existed. In particular, we would find $\{\alpha\}=\sigma_p(\alpha S_1)=\sigma_p(\beta S_2)=\{\beta\}$, where $\sigma_p$ denotes the point spectrum. This is absurd unless
$\alpha=\beta$, in which case the corresponding endomorphism does extend being but a gauge automorphism.
\end{proof} 
It is no surprise that the same proof as above covers the case of the so-called anti-diagonal automorphisms. These are simply given  by $\rho_{\alpha,\beta}(S_1)=\alpha S_2$ and $\rho_{\alpha,\beta}(S_2)=\beta S_1$ for any given $\alpha,\beta\in\mathbb{T}$. Again,  an automorphism $\rho_{\alpha,\beta}$ extends precisely when $\alpha=\beta$. 
To complete the picture, we shall presently determine which Bogoljubov automorphisms of $\O_2$ extend to endomorphisms of $\Q_2$. A suitable adaptation of some of the techniques developed by Matsumoto and Tomiyama in \cite{Matsumoto} will be again among the ingredients  to concoct the proof of the main result of this section. 
This says that the extensible Bogoljubov automorphisms are precisely the flip-flop, the gauge automorphisms, and their products, which altogether form a group isomorphic with the direct product $\mathbb{T}\times\mathbb{Z}_2$. To this aim, consider a unitary matrix
\begin{eqnarray*}
	A=\left( \begin{array}{cc} 
 		a & b\\
		c & d\\
		 		\end{array}\right)\in U(\mathbb{C}^2) .
\end{eqnarray*}
and let $\alpha =\lambda_A$ be the corresponding automorphism of $\O_2$, i.e. $\alpha(S_1)=aS_1+cS_2=(aU+c)S_2$ and $\alpha(S_2)=bS_1+dS_2=(bU+d)S_2$. Set $f(U)=(bU+d)$ for short.
 The condition $S_2U=U^2S_2$ implies that $f(U) S_2 \alpha(U)=\alpha(U)^2 f(U) S_2$. Suppose that $\alpha$ is extensible and denote by $\widetilde{\alpha}$ such an extension. Finally, set $\tilde{U}\doteq \widetilde{\alpha}(U)$, $\widetilde{S}_2\doteq \alpha(S_2)$, $\widetilde{S}_1\doteq \alpha(S_1)$. 
From now on we shall always be focusing on the case where $a,b,c,d$ are all different from zero. That said, the first thing we need to prove is the extension is unique provided that it exists.  

\begin{lemma}
If $\lambda_A$ extends, then its extension is unique. 
\end{lemma}
\begin{proof}
By Proposition \ref{prop-uniq-ext} all we need to check is that $\tilde{S}_1$ is pure as an isometry acting on $\ell_2(\mathbb{Z})$. 
This entails ascertaining that $\bigcap_n\textrm{Ran}[\tilde{S}_1^n(\tilde{S}_1^*)^n]=\{0\}$. To this aim, let us set
$M_n\doteq\textrm{Ran}[\tilde{S}_1^n(\tilde{S}_1^*)^n]$. As $M_{n+1}\subset M_n$, we have that $E_{\bigcap_n M_n}=\lim E_{M_n}$ strongly. Thus
we are led to show $\lim_n E_{M_n}=0$. For this it is enough to prove $\lim_n \|E_{M_n}e_k\|= 0$ for every $k$. Now the powers of
$\tilde{S}_1$ are given by $\tilde{S}_1^n=\sum_{|\alpha |= n} c_\alpha S_\alpha$, where $c_\alpha\doteq a^{n_1(\alpha)}c^{n_2(\alpha)}\in \mathbb{C}^*$ with $n_1(\alpha)$ being the number of $1$'s occurring in $\alpha$ and  $n_2(\alpha)$ the number of $2$'s occurring in $\alpha$. We set $L\doteq\max \{|a |, |c|\}$ and observe that $L<1$ by the hypotheses on the unitary matrix $A$. We have that 
$$
\|(\tilde{S}_1^*)^ne_k \|=|c_{\alpha(k)}| \leq L^{n}\to 0 \quad n\to+\infty 
$$
for a unique coefficient $c_{\alpha(k)}$ that depends on $k$ (this is actually a consequence of the fact that $\sum_{|\alpha |=k}S_\alpha S_\alpha^*=1$).
This in turn implies the claim.
\end{proof}
In light of the previous result, it is a very minor abuse of notation to denote by $\lambda_A$ also its extension to $\Q_2$  when existing. 
\begin{lemma}
If $\lambda_A$ extends, then $\tilde{U}\in\Q_2^{\mathbb{T}}$. 
\end{lemma}
\begin{proof}
Suppose that $\tilde{U}$ is not in $\Q_2^{\mathbb{T}}$. Then by definition there must exist a non-trivial gauge automorphism $\tilde{\alpha_z}$ such that $\tilde{\alpha_z}(\tilde{U})\neq \tilde{U}$. 
By applying $\tilde {\alpha_z}$ to both sides of the equalities $\tilde{S}_2\tilde{U}=\tilde{U}^2\tilde{S}_2$ and $\tilde{S}_2\tilde{S}_2^*+\tilde{U}\tilde{S}_2\tilde{S}_2^*\tilde{U}^*=1$, we also get
\begin{align*}
& \tilde{\alpha_z}(\tilde{U})^2 \tilde{S}_2=\tilde{S}_2\tilde{\alpha_z}(\tilde{U})\\
& \tilde{S}_2\tilde{S}_2^*+\tilde{\alpha_z}(\tilde{U}) \tilde{S}_2\tilde{S}_2^* \tilde{\alpha_z}(\tilde{U})^*=1\\
\end{align*}
which together say  there exists an endomorphism $\Lambda\in\End(\Q_2)$ such that $\Lambda(S_2)=\tilde{S}_2$ and $\Lambda(U)=\tilde{\alpha_z}(\tilde{U})$. 
Now $\Lambda(S_1)=\Lambda(US_2)=\Lambda(U)\Lambda(S_2)=\tilde{\alpha_z}(\tilde{U})\tilde{S}_2= \bar{z}\tilde{\alpha_z}(\tilde{U}\tilde{S}_2)=\bar{z}\tilde{\alpha_z}(\tilde{S}_1)= \tilde{S}_1=\lambda_A(S_1)$.
A contradiction is thus arrived at because $\Lambda$ and $\tilde{\lambda_A}$ are different maps. 
\end{proof}
Now we introduce some lemmas to prove that $\tilde\alpha(U)$ is actually contained in $C^*(U)$.

\begin{lemma}\label{lemma-3-R-2}
For any $x\in \B_2^k$, the element $(\widetilde{S}_2^*)^k x \widetilde{S}_1^k$ belongs to $C^*(U)$.
\end{lemma}
\begin{proof}
Suppose that $x=S_\alpha S_\beta^* U^h$, where $|\alpha|=|\beta|=k$. If $h\geq 0$, we have that
$(\widetilde{S}_2^*)^k x \widetilde{S}_1^k = (\widetilde{S}_2^*)^k S_\alpha S_\beta^* U^h \widetilde{S}_1^k$ is a polynomial in $U$.  
The case $h\leq 0$ can be handled with similar computations.
\end{proof}

\begin{lemma}\label{lemma-4-R-2}
Let $x\in \Q_2^\mathbb{T}$ such that the sequence $(\widetilde{S}_2^*)^k x \widetilde{S}_1^k$ converges to an element $z$. Then $z\in C^*(U)$.
\end{lemma}
\begin{proof}
Let $\{ y_k \}_{k\geq 0}$ be a sequence such that $y_k\in\B_2^k$ and $y_k\to x$ normwise. Then the thesis follows from the following inequality
$$
\| z - (\widetilde{S}_2^*)^k y_k \widetilde{S}_1^k \|\leq \| z - (\widetilde{S}_2^*)^k x \widetilde{S}_1^k \|+ \| (\widetilde{S}_2^*)^k (x-y_k) \widetilde{S}_1^k \|\; .
$$
\end{proof}

\begin{lemma}\label{lemma-5-R-2}
We have that $\tilde{U}\in C^*(U)$.
\end{lemma}
\begin{proof}
By applying $\widetilde{\lambda_A}$ to the identity $US_1^k =S_2^k U$ we get $ \tilde{U} \tilde{S}_1^k=\tilde{S}_2^k\tilde{U}$.
For  all $k\in \mathbb{N}$ we have that 
$(\tilde{S}_2^*)^k\tilde{U}\tilde{S}_1^k=\tilde{U}$.
Therefore, $\tilde{U}$ is in $C^*(U)$ thanks to Lemma \ref{lemma-4-R-2}, applied to
 $x=\tilde{U}$.
 \end{proof}

We have verified that $\alpha(U)=g(U)$ for some $g\in C(\mathbb{T})$, which turns out to be vital in proving the following result.
\begin{theorem}\label{Bogoljubov}
If $\alpha\in \Aut(\O_2)$ is a Bogoljubov automorphism, then $\alpha$ extends to $\Q_2$ if and only if $\alpha$ is  the flip-flop, a gauge automorphism, or a composition of these two.
\end{theorem}
\begin{proof}
By the discussion at the beginning of this section it is enough to consider the case in which $a,b,c,d$ are all different from zero.
All the computations are henceforth made in the canonical representation. The condition $f(U) S_2 \alpha(U)=\alpha(U)^2 f(U) S_2$ yields 
\begin{align*}
f(U) S_2 g(U) & = g(U)^2 f(U) S_2\\
f(U) g(U^2) S_2 & = f(U) g(U)^2 S_2 .
\end{align*}
Since the point spectrum of $U$ is empty, $f(U)$ is always injective, unless $b=d=0$, in which case $A$ is not unitary. 
Thus $g(U^2) S_2  = g(U)^2 S_2$. At the function level we must then have $g(z^2)=g(z)^2$ for every $z\in\mathbb{T}$. By continuity, we find that $g(z)=z^l$, for this see e.g. the Appendix. Therefore $g(U)=U^{l}$. 
We have that $\alpha(S_1)=aS_1+c S_2= bU^{l+1} S_2+d U^{l}S_2$. If we compute the above equality on the vectors $e_m$, we get
$$
a e_{2m+1}+c e_{2m} = b e_{2m+l+1}+ d e_{2m+l} .
$$
which is to be satisfied for each $m\in\mathbb{Z}$. Therefore, there are only two possibilities to fulfill these conditions:
\begin{enumerate}
\item $l=1$, and $a=d\neq 0$, $b=c=0$;
\item $l=-1$, and $b=c\neq 0$, $a=d=0$.
\end{enumerate}
The first corresponds to gauge automorphisms, whilst  the second to the flip-flop and its compositions with gauge automorphisms.

\end{proof}

\section{Outer automorphisms}\label{Outerness}
In this section the group $\out(\Q_2)$ is shown to be non-trivial. More precisely, it turns out to be a non-abelian uncountable group. A thorough description of its structure, though, is still missing. As far as  we know, it might  well be chimerical to get. 

\subsection{Gauge automorphisms and the flip-flop}
Below the flip-flop and non-trivial gauge automorphisms are proved to be outer. In fact,  this parallels analogue known results for $\O_2$. Since gauge automorphisms are more easily dealt with, we start our discussion focusing
on them first.  The next result (partly) follows from Proposition \ref{zS}, however we give an alternative proof because it will shed further light on the properties of the gauge automorphisms, see Remark \ref{notweakcont}.

\begin{theorem}\label{GaugeOut}
The extensions to $\Q_2 $ of the non-trivial gauge automorphisms of $\O_2$ 
are still outer automorphisms
(and they are not weakly inner in the canonical representation). 
\end{theorem}
\begin{proof}
We shall argue by contradiction. From now on $\Q_2$ will always be thought of as a concrete subalgebra of $B(\ell_2(\mathbb{Z}))$ via the canonical representation. We will actually prove: 
no unitary $V\in B(\ell_2(\mathbb{Z}))$ can implement a non-trivial gauge automorphism.
Indeed, let $z\in\mathbb{T}$ different from $1$, and  let $\Lambda_z\in\Aut(\Q_2)$ the corresponding gauge automorphism. If $V$ is a unitary operator on $\ell_2(\mathbb{Z})$ such that 
$\Lambda_z=\textrm{ad}(V)\upharpoonright_{\Q_2}$, then in particular we must have $\Lambda_z(U)=\textrm{ad}(V)(U)$, namely $U=VUV^*$. This shows that $V$ commute with $U$. Since $U$ generates a \emph{MASA}, $V$ must take the form $V=f(U)$, for some $f\in L^{\infty}(\mathbb{T})$; in particular it belongs to $W^*(U)$ too. But we also have   $zS_2=\Lambda_z(S_2)=\textrm{ad}(V)(S_2)=VS_2V^*$, that is to say $zS_2V=VS_2$. If we now compute this identity between operators on the vector $e_0$, we get $zS_2Ve_0=Ve_0$, i.e. 
$S_2Ve_0=\bar zVe_0$. As a consequence, we also have $Ve_0=0$, merely because $1$ is the only eigenvalue of $S_2$. Since
$V\in W^*(U)$ and $e_0$ is a separating vector for $W^*(U)$ as well, we finally find that $V$ is zero, which is clearly a contradiction.
\end{proof}
\begin{remark}\label{notweakcont}
Actually, the proof given above says a bit more. Indeed, the non-trivial gauge automorphisms of $\Q_2\subset B(\ell_2(\mathbb{Z}))$  are not weakly continuous. For if they were, they should clearly extend to an automorphism of $B(\ell_2(\mathbb{Z}))$, but this is absurd
because the automorphisms of $B(\ell_2(\mathbb{Z}))$ are all inner.
\end{remark}
Among other things, we also gain the additional information that $\out(\Q_2)$ is an uncountable group, in that different gauge automorphisms give raise to distinct classes in $\out(\Q_2)$. Indeed, if $\tilde\alpha_z$ and $\tilde\alpha_w$ are two different gauge automorphisms, then by Proposition \ref{zS} there cannot exist any unitary $u\in\Q_2$ such that $u\tilde\alpha_z(x) u^*=\tilde\alpha_w(x)$ for every $x\in\Q_2$. 
For the sake of completeness we should also mention that every separable traceless $C^*$-algebra is actually known to have uncountable many outer automorphisms \cite{Phillips}. 
\begin{remark}
Notably, the former result also provides a new and simpler proof of the well-known fact that the gauge automorphisms on $\O_2$
are outer. However, the case of a general $\O_n$ cannot be recovered from our discussion, and must needs be treated separately, as already done elsewhere. 
\end{remark}

As for the flip-flop, instead, we start our discussion by showing it is a weakly inner automorphism, which is the content of the next 
result.
\begin{proposition}
The extension of the flip-flop to $\Q_2$ is a weakly inner automorphism. 
\end{proposition}
\begin{proof}
By definition, we only have to produce a representation $\pi$ of $\Q_2$ such that
$\tilde\lambda_f$ is implemented by a unitary in $\pi(\Q_2)''$. The canonical representation does this job well. For if $V\in U(\ell_2(\mathbb{Z}))$ is the self-adjoint unitary given by $Ve_k\doteq e_{-k-1}$, the equalities $VS_1V^*=S_2$ and $VS_2V^*=S_1$ are both easily checked. Since the canonical representation is irreducible, the proof is thus complete.
 \end{proof}
This result should also be compared with a well-known theorem by Archbold \cite{Archbold} that the flip-flop is weakly inner on $\O_2$.
\begin{remark}
The unitary $V$ as defined above can be rewritten as $V=\mathcal{P}U=U^*\mathcal{P}$, where $\mathcal{P}$ is the self-adjoint unitary given by $\mathcal{P}e_k=e_{-k}$, $k\in\mathbb{Z}$.
Obviously, $V$ is in $\Q_2$ if and only if $\mathcal{P}$ is. We shall prove that $\mathcal{P}$ is not in $\Q_2$ in a while. At any rate, we observe the equality $V(S_1VS_1^*+S_2VS_2^*)^*=S_1S_2^*+S_2S_1^*\doteq f\in\O_2$, which is immediately checked, and  $fV=S_1VS_1^*+S_2VS_2^*=Vf$. 
Finally, it is worth noting that $U=\lambda_{\mathbb{Z}}(1)$, and that $\mathcal{P}$ is the canonical intertwiner between $\lambda_{\mathbb{Z}}$ and $\rho_{\mathbb{Z}}$. In this picture, $\Q_2$ is thus the concrete C*-algebra on $\ell_2(\mathbb{Z})$ generated by $C^*_r(\mathbb{Z})$ and the copy of $\O_2$ provided by the canonical representation. 
\end{remark}
In spite of being weakly inner, $\tilde\lambda_f$ is an outer automorphism, as is its restriction to $\O_2$. To prove that, we first need to show that the unitary $V$ above is up to multiplicative scalars the unique operator in $B(\ell_2(\mathbb{Z}))$ that implements the flip-flop.   
\begin{proposition}
If $W\in B(\ell_2(\mathbb{Z}))$ is such that ${\rm ad}(W)\upharpoonright_{\Q_2}=\tilde\lambda_f$, then $W=\lambda V$ for some $\lambda\in\mathbb{T}$.
\end{proposition}
\begin{proof}
First note that we must have  ${\rm ad}(W^2)=id_{B(\ell_2(\mathbb{Z}))}$ since the flip-flop is an involutive automorphism and $\Q_2''=B(\ell_2(\mathbb{Z}))$. Hence $W^2$ is a multiple of $1$, and therefore there is no loss of generality if we also assume
that $W^2=1$, i.e. $W=W^*$. From the relation $WS_1W=S_2$, we get $S_1We_0=We_0$. Hence $We_0=\lambda e_{-1}$ for some $\lambda\in\mathbb{T}$. From $e_0=W^2e_0=\lambda W e_{-1}$ we get $We_{-1}=\bar\lambda e_0$. We now show that either $\lambda=1$ or $\lambda=-1$. Indeed, from $UW=WU^*$ it follows that $We_{-1}=WU^*e_0=UWe_0=U(\lambda e_{-1})=\lambda e_0=\bar\lambda e_0$, which in turn implies that $\lambda$ is real. Of course, we only need to deal with the case $\lambda=1$. The conclusion is now obtained at once if we use the equality $WU=U^*W$ inductively, for $We_{k+1}=WUe_k=U^*We_k=U^*e_{-k-1}=e_{-k-2}$, as maintained.
\end{proof}
\begin{remark}
Of course, the uniqueness of $V$ could also have been obtained faster merely by irreducibility of $\Q_2$. However, the proof displayed above has the advantage of showing how we came across the operator $V$.
\end{remark}
Here finally follows the theorem on the outerness of the flip-flop.
\begin{theorem}\label{FlipFlopOut}
The extension of the flip-flop is an outer automorphism.
\end{theorem}
\begin{proof}
Thanks to the former result, all we have to prove is that $\mathcal{P}$ is not in $\Q_2$, which entails checking that $\mathcal{P}$ cannot be a norm limit of a sequence ${x_k}$ of operators of the form $x_k=\sum_k c_k S_{\alpha_k} S_{\beta_k^*}U^{h_k}$. Indeed, if this were the case, we should have $\varepsilon >\|\P-\sum_{\alpha,\beta, h} c_{\alpha,\beta,h}S_\alpha S_\beta^*U^h\|$  for some finite sum of the kind $\sum_{\alpha, \beta, h} c_{\alpha,\beta,h}S_\alpha S_\beta U^h$, with $\varepsilon>0$ as small as needed. If so, we would also find the inequality  
$$
\|e_{-n}-\sum_{\alpha, \beta, h}c_{\alpha,\beta,h}S_\alpha S_\beta^*U^he_n\|= \|\mathcal{P}e_n-\sum_{\alpha, \beta, h}c_{\alpha,\beta,h}S_\alpha S_\beta^*U^he_n\|<\varepsilon
$$ 
This inequality, though, becomes absurd as soon as $\varepsilon< 1$ and $n$ is sufficiently large, i.e. $n$ is bigger than the largest value of $|h|$, for we would have $\|e_{-n}-\sum_{\alpha, \beta, h}c_{\alpha,\beta,h}S_\alpha S_\beta^*U^he_n\|^2=1+\|\sum_{\alpha, \beta, h}c_{\alpha,\beta,h}S_\alpha S_\beta^*U^he_n\|^2\geq 1$, as $\sum_{\alpha, \beta, h}c_{\alpha,\beta,h}S_\alpha S_\beta^*U^he_n\in \H_+$.
\end{proof}
Now, as we know $\tilde\lambda_f$ is outer, we would also like to raise the question whether there exists a representation of $\Q_2$ in which $\tilde\lambda_f$ is not unitarily implemented. The answer would indeed complete our knowledge of $\tilde\lambda_f$ itself.

\subsection{A general result}
We saw above that the flip-flop is an outer automorphism. This is not an isolated case, for every automorphism that takes $U$ to its adjoint must in fact be outer. This is the content of the next result.
\begin{theorem}\label{GenOut}
Every automorphism $\alpha\in\Aut(\Q_2)$ such that $\alpha(U)=U^*$ is an outer automorphism.
\end{theorem}
\begin{proof}
All we have to prove is that there is no unitary $W\in\Q_2$ such that $WUW^*=U^*$. To this aim, we will be working in the canonical representation. If $W\in B(\H)$ is a unitary operator such that $WUW^*=U^*$, then we have 
$WUW^*=\mathcal{P}U\mathcal{P}$, hence $\mathcal{P}WU(\mathcal{P}W)^*=U$, which says that $\mathcal{P}W$ commutes with $U$. Therefore, $\mathcal{P}W=f(U)$ for some $f\in L^\infty(\mathbb{T})$ with $|f(z)|=1$ a.e. with respect to the Haar measure of $\mathbb{T}$, hence $W=\mathcal{P}f(U)$. Then we need to show that such a $W$ cannot be in $\Q_2$. If $f$ is a continuous function, there is not much to say, for $\mathcal{P}f(U)\in\Q_2$ would immediately imply that $\mathcal{P}=\mathcal{P}f(U)f(U)^*$ is in $\Q_2$ as well, which we know is not the case. The general case of an essentially bounded function is dealt with in much the same way apart from some technicalities to be overcome. Given any $f\in L^\infty(\mathbb{T})$ and $\varepsilon>0$, thanks to Lusin's theorem  we find a closed set $C_\varepsilon\subset\mathbb{T}$ such that $\mu(\mathbb{T}\setminus C_\varepsilon)<\epsilon$ and $f\upharpoonright_{C_\varepsilon}$ is continuous. This in turn guarantees that there exists a continuous function $g_\varepsilon\in C(\mathbb{T}, \mathbb{T})$ that coincides with $f$ on $C_\varepsilon$ by an easy application of the Tietze extension theorem. If $\mathcal{P}f(U)$ is in $\Q_2$, then $\mathcal{P}f(U)g_\varepsilon(U)^*$ is also in $\Q_2$. Note that $f\bar{g_\varepsilon}=1+h_\varepsilon$, where $h_\varepsilon$ is a suitable function whose support is contained in $\mathbb{T}\setminus C_\varepsilon$. In particular, we can rewrite  $\mathcal{P}f(U)g_\varepsilon(U)^*$ as $\mathcal{P}+\mathcal{P}h_\varepsilon(U)$. If the latter operator were in $\Q_2$, then we could find an operator of the form $\sum_{\alpha, \beta, h} c_{\alpha,\beta, h}S_\alpha S_\beta^*U^h$ such that $\|\mathcal{P}+\mathcal{P}h_\varepsilon(U)-\sum_{\alpha, \beta, h} c_{\alpha,\beta, h}S_\alpha S_\beta^*U^h\|<\varepsilon$.
If $N$ is any natural number greater than the maximum value of $|h|$ as $h$ runs over the set the above summation is performed on, we should have $\|\mathcal{P}e_N+\mathcal{P}h_\varepsilon(U)e_N-\sum_{\alpha, \beta, h} c_{\alpha,\beta, h}S_\alpha S_\beta^*U^he_N\|<\varepsilon$, namely $$\|e_{-N}+\mathcal{P}h_\varepsilon(U) e_N-\sum_{\alpha, \beta, h} c_{\alpha,\beta, h}S_\alpha S_\beta^*U^h e_N \|<\varepsilon$$ 
But then we should also have
$$\|e_{-N}+\mathcal{P}h_\varepsilon(U) e_N-\sum_{\alpha, \beta, h} c_{\alpha,\beta, h}S_\alpha S_\beta^*U^h e_N \|\geq
\|e_{-N}-\sum_{\alpha, \beta, h} c_{\alpha,\beta, h}S_\alpha S_\beta^*U^h e_N\|-\|\mathcal{P}h_\varepsilon(U) e_N\|$$
Hence
$$\varepsilon >\|e_{-N}+\mathcal{P}h_\varepsilon(U) e_N-\sum_{\alpha, \beta, h} c_{\alpha,\beta, h}S_\alpha S_\beta^*U^h e_N \|\geq
1-\|\mathcal{P}h_\varepsilon(U) e_N\|$$
The conclusion is now got to if we can show that  $\|\mathcal{P}h_\varepsilon(U) e_N\|$ is also as small as needed.
But the norm $\|\mathcal{P}h_\varepsilon(U) e_N\|$ is easily computed in the Fourier transform of the canonical representation, where it takes the more workable form $(\int|h_\varepsilon(z^{-1})z^{-N}|^2\rm {d}\mu(z))^{\frac{1}{2}}$ and is accordingly smaller than $2\mu(\mathbb{T}\setminus C_\varepsilon)^{\frac{1}{2}}\leq 2\varepsilon^{\frac{1}{2}}$. The above inequality becomes absurd as soon as $\varepsilon$ is taken small enough.
\end{proof}
As a straightforward consequence, we can immediately see that the class of the flip-flop in $\out(\Q_2)$ do not coincide with any of the classes of the gauge automorphisms. In other terms, the automorphisms
$\widetilde\alpha_z\circ\widetilde\lambda_f^{-1}$ are all outer, sending $U$ in $U^*$. Furthermore, if we now denote by $\chi_{-1}$ the automorphism such that
$\chi_{-1}(S_2)\doteq S_2$ and $\chi_{-1}(U)\doteq U^{-1}=U^*$, then the above result also applies to $\chi_z\doteq\chi_{-1}\circ\widetilde\alpha_z$, which are outer as well by the same reason. Also note that $\chi_z(S_2)=zS_2$.
More interestingly, if $z$ is not $1$, the corresponding $\chi_z$ yields a class in $\out(\Q_2)$ other than the one of the flip-flop. To make sure this is true, we start by noting that $\chi_{-1}$ and $\widetilde\lambda_f$ do not commute with each other.  
Even so, they do commute in $\out(\Q_2)$, in that they even yield the same conjugacy class. Indeed, we have ${\rm ad}(U^*)\circ\chi_{-1}=\widetilde\lambda_f$, or equivalently  $\widetilde\lambda_f\circ \chi_{-1}= {\rm ad}(U^*)$, which in addition says that $\widetilde\lambda_f\circ \chi_{-1}$ has infinite order in $\Aut(\Q_2)$ while being the product of two automorphisms of order $2$. From this our claim follows easily. For, if $\chi_z\circ\widetilde\lambda_f$ is an inner automorphism, the identity $\chi_z\circ\widetilde\lambda_f=\widetilde\chi_z\circ\chi_{-1}\circ\widetilde\lambda_f=\widetilde\chi_z\circ\rm{ad}(U)$ implies at once that  $\widetilde\chi_z$ is inner as well, which is possible for $z=1$ only.
However, the classes $[\chi_z]$ and $[\widetilde\lambda_f]$ do commute in $\out(\Q_2)$, because $\chi_z\circ\widetilde\lambda_f\circ\chi_z^{-1}\circ\widetilde\lambda_f=\rm{ad}(U^2)$. In order to prove that $\out(\Q_2)$ is not abelian, we still need to sort out a new class of outer automorphisms. This will be done in the next sections.

\section{Notable endomorphisms and automorphism classes}\label{SecAut}

\subsection[Endomorphisms and automorphisms $\alpha$ such that $\alpha(S_2)=S_2$]{Endomorphisms and automorphisms $\pmb{\alpha}$ such that $\pmb{\alpha(S_2)=S_2}$}
For any odd integer $2k+1$, whether it be positive or negative, the pair $(S_2, U^{2k+1})$ still satisfies the two defining relations of $\Q_2$. This means
that the map that takes $S_2$ to itself and $U$ to $U^{2k+1}$ extends to an endomorphism of $\Q_2$, which will be denoted by $\chi_{2k+1}$. Trivially, this endomorphism extends the identity automorphism of $C^*(S_2)$.
A slightly less obvious thing to note is that these endomorphisms cannot be obtained as extensions of endomorphisms of $\O_2$. Indeed, $\chi_{2k+1}(S_1)=\chi_{2k+1}(US_2)=U^{2k+1}S_2$, and $U^{2k+1}S_2$ is not in $\O_2$: if it were, we would also find that $S_1^*U^{2k+1}S_2=S_2^*U^*UU^{2k}S_2=S_2^*S_2U^k=U^k$ would be in $\O_2$, which is not.
In this way we get a class endomorphisms $\chi_{2k+1}$, $k \in {\mathbb Z}$ with $\chi_1={\rm id}$ and $\chi_{-1}$ being clearly an automorphism of order two. All these endomorphisms commute with one other and we have $\chi_{2k+1} \circ \chi_{2h+1} = \chi_{(2k+1)(2h+1)}$ for any $k,h \in {\mathbb Z}$. Phrased differently, the set $\{\chi_{2k+1}: k\in\mathbb{Z}\}$ is a semigroup of proper endomorphisms of
$\Q_2$.
One would like to know if the endomorphisms singled out above give the complete list of the endomorphisms of $\Q_2$ fixing $S_2$. In other words, the question is  whether the set
$$\U_2 \doteq  \{V \in U({\mathcal Q}_2) \ | \  V^2 S_2 = S_2 V, \: S_2 {S_2}^* + V S_2 S_2^* V^* = 1 \}$$
contains elements other than the $U^{2k+1}$ with $k \in {\mathbb Z}$ above. As a matter of fact, answering this question in its full generality is not an easy task. An interesting if partial result does surface, though, as soon as we introduce an extra assumption.   
Going back to the endomorphisms $\chi_{2k+1}$, we next show they are all proper apart from $\chi_{\pm 1}$. The proof cannot be considered quite elementary, in that it uses the maximality of $C^*(U)$. 

\begin{proposition}
None of the endomorphisms $\chi_{2k+1}$ is surjective if $2k+1\neq \pm 1 $.
\end{proposition}
\begin{proof}
Let $\mathfrak{A}\subset\Q_2$ the $C^*$-subalgebra of those $x\in \Q_2$ such that $\chi_{2k+1}(x)\in C^*(U)$. We clearly have $C^*(U)\subset\mathfrak{A}$. By simplicity of
$\Q_2$ the endomorphism $\chi_{2k+1}$ is injective, which means $\mathfrak{A}$ is still commutative. Therefore, by maximality of $C^*(U)$, we must have $\mathfrak{A}=C^*(U)$. 
From this it now follows that $U$ is not in the range of $\chi_{2k+1}$, for the restriction $\chi_{2k+1}\upharpoonright_{C^*(U)}$ is induced, at the spectrum level, by the map $\mathbb{T}\ni z\mapsto z^{2k+1}\in\mathbb{T}$.
\end{proof}

In addition, in the canonical representation the endomorphisms $\chi_{2k+1}$ cannot be implemented by any unitary $W\in\B(\ell_2(\mathbb{Z}))$. Indeed, we can state the following result.

\begin{proposition} 
Let $2k+1$ be an odd integer different from 1. Then there is no unitary $V$ in $B(\ell_2(\mathbb{Z}))$ such that $V S_2 = S_2 V$ and $VUV^* = U^{2k+1}$.
\end{proposition}
\begin{proof}
The proposition is easily proved by \emph{reductio ad absurdum}. Let $V$ be such a unitary as in the statement. 
From $VS_2 e_0 = S_2 V e_0$ we deduce that $Ve_0$ is an eigenvector of $S_2$ with eigenvalue 1.
Without loss of generality, we may assume that $Ve_0 = e_0$. Now,
$$V e_h = V U^h e_0 = U^{(2k+1)h} V e_0 = U^{(2k+1)h} e_0 = e_{(2k+1)h}, \ h=0,1,2,\ldots  $$
and, similarly,
$$V e_{-h} = e_{-(2k+1)h}, \ h=-1, -2, \ldots \ . $$
To conclude, it is now enough to observe that $V$ is not surjective whenever $2k+1\neq -1$, whereas the case of $2k+1=-1$ leads to $V$ being equal to  $\mathcal{P}$, which does not belong to $\Q_2$.
\end{proof}
Rather than saying what the whole $\U_2$ is, we shall focus on its subset $\U_2\cap C^*(U)$ instead, which is more easily dealt with.
This task is accomplished by the next result.
\begin{theorem}
The set $\U_2$\scalebox{1.5}{$\cap$}$C^*(U)$ is exhausted by the odd powers of $U$, i.e.   $\U_2$ \scalebox{1.5}{$\cap$} $C^*(U)=\{U^{2k+1}: k\in\mathbb{Z}\}$.
\end{theorem}
\begin{proof}
Let $W\in \U_2$\scalebox{1.5}{$\cap$}$ C^*(U)$. Then there exists a function $f\in C(\mathbb{T},\mathbb{T})$ such that $W=f(U)$. The condition $S_2W=W^2S_2$ can be rewritten as $S_2f(U)=f(U)^2S_2$. On the other hand, we also have $S_2f(U)=f(U^2)S_2$. Therefore
$f(U^2)S_2=f(U)^2S_2$, that is $f(U^2)=f(U)^2$. Accordingly, the function $f$ must satisfy the functional equation $f(z^2)=f(z)^2$. Being continuous,  our function $f$ must be of the form $f(z)=z^n$ for some integer $n\in\mathbb{Z}$, see the Appendix. This means that $W=U^n$. If we also impose the  condition on the ranges
$S_2S_2^*+U^nS_2S_2^*U^{-n}=1$, we finally find that $n$ is forced to be an odd number, say $n=2k+1$.
\end{proof}
The result obtained above can also be stated in terms of endomorphisms of $\Q_2$. With this in mind, we need to introduce a bit of notation. In particular,
we denote by $\End_{C^*(S_2)}(\Q_2, C^*(U))$ the semigroup of those endomorphisms of $\Q_2$ that fix $S_2$ and leave $C^*(U)$ globally invariant.
\begin{corollary}
The semigroup $\End_{C^*(S_2)}(\Q_2, C^*(U))$ identifies with $\{\chi_{2k+1}: k\in\mathbb{Z}\}$. As a result, we also have
$$
\Aut_{C^*(S_2)}(\Q_2, C^*(U))=\{\textrm{id},\chi_{-1}\}\cong\mathbb{Z}_2
$$
\end{corollary}
The foregoing result might possibly be improved by dropping the hypothesis that our endomorphisms leave $C^*(U)$ globally invariant also. This is in fact
a problem we would like to go back to elsewhere.

\subsection[Automorphisms $\alpha$ such that $\alpha(U)=U$]{Automorphisms $\pmb{\alpha}$ such that $\pmb{\alpha(U)=U}$}
In this section we study those endomorphisms and automorphisms $\Lambda\in\End(\Q_2)$ such that $\Lambda(U)=U$. Of course, the problem of describing all of them amounts to determining the set $$\mathcal{S}_2\doteq\{W\in Q_2: W^*W=1, WU=U^2W, WW^*+UWW^*U^*=1\}  \ . $$
Curiously enough, it turns out that $\S_2$ can be described completely, which is what this section is chiefly aimed at. We start our discussion by sorting out quite a simple class of automorphisms of that sort. Given a function $f\in C(\mathbb{T},\mathbb{T})$ we denote by $\beta^f$ the automorphism of $\Q_2$ given by $\beta^f(U)=U$ and $\beta^f(S_2)=f(U)S_2$, which is well defined because the pair $(f(U)S_2, U)$ still satisfies the two defining relations of $\Q_2$. Note that $\beta^f \circ \beta^g = \beta^{f \cdot g}$ so that we obtain an abelian subgroup of ${\rm Aut}_{C^*(U)}(\Q_2)$ and that a constant function $f(z)=w$ gives back the gauge automorphism $\widetilde\alpha_w$.
Furthermore, we have the following result, which gives a sufficient condition on $f$ for the corresponding $\beta^f$ to be outer. As the condition is not at all restrictive, the correspondence 
$f\mapsto\beta^f$, which is one to one, provides plenty of outer automorphisms.
\begin{proposition}
If $f\in C(\mathbb{T},\mathbb{T})$ is such that $f(1)\neq 1$, then $\beta^f$ is an outer automorphism.
\end{proposition}
\begin{proof}
If $V\in\Q_2$ is a unitary such that $\beta^f=\textrm{ad}(V)$, then $V$ commutes with $U$ and therefore it is of the form $g(U)$ for some $g\in C(\mathbb{T},\mathbb{T})$ by maximality of $C^*(U)$.
The condition $\beta^f(S_2)=\textrm{ad}(V)(S_2)$ yields the equation $f(U)S_2=g(U)S_2g(U)^*$, that is $f(U)S_2=g(U)g(U^2)^*S_2$. But then $g$ and $f$ satisfy the relation $f(z)=g(z)\overline{g(z^2)}$ for every $z\in\mathbb{T}$.
In particular, the last equality says that $f(1)=g(1)\overline{g(1)}=1$.
\end{proof}

However, the condition spotted above is not necessary. This will in turn result as a consequence of the following discussion. We will be first concerned with the problem as to whether an automorphism $\beta^f$ may be equivalent in $\out(\Q_2)$ to a gauge automorphism. If so, there exist $z_0\in\mathbb{T}$ and  $W\in\U(\Q_2)$ such that $WUW^*=U$ and $Wf(U)S_2W^*=z_0S_2$. As usual, the first relation says that $W=h(U)$ for some $h\in C(\mathbb{T},\mathbb{T})$, which makes the second into $h(U)f(U)h(U^2)^*=z_0$, that is $h$ satisfies the functional equation $h(z)f(z)\overline{h(z^2)}=z_0$. The latter says in particular that $z_0$ is just $f(1)$.
We next show that there actually exist  many continuous functions $f$ for which there is no continuous $h$ that satisfies 
\begin{align} \label{funct-eq-gauge}
\overline{h(z)}h(z^2)=f(z)\overline{f(1)}\doteq\Psi(z) \; .
\end{align}
Note that $\Psi(1)=1$ and that  there is no loss of generality if we assume $h(1)=1$ as well. By evaluating \eqref{funct-eq-gauge} at $z=-1$ we find $h(-1)=1$ provided that $f(-1)=f(1)$.
\begin{remark}
By density, the continuous solutions of the equation $\overline{h(z)}h(z^2)=\Psi(z)$ are completely determined by the values they take at the $2^n-$th roots of unity. Furthermore, the value of such an $h$ at a point $z$ with $z^{2^n}=1$ is simply given by the interesting formula $h(z)=\frac{1}{\prod_{k=0}^{n-1}\Psi(z^{2^k})}$. The latter is easily got to by induction starting from the relation $h(z^2)=h(z)\Psi(z)$.
\end{remark}
Here is our result, which provides examples of $\beta^f$ not equivalent with any of the gauge automorphisms. Let $f\in C(\mathbb{T},\mathbb{T})$ be such that $f(e^{i\theta})=1$ for $0\leq \theta\leq\pi$ and $f(e^{i\theta})=-1$ for $\pi+\varepsilon\leq\theta\leq 2\pi-\varepsilon$ with $0<\varepsilon\leq\frac{\pi}{4}$, then we have the following.
\begin{proposition}
If $f\in C(\mathbb{T},\mathbb{T})$ is a function as above, then the associated $\beta^f$ is not equivalent to any gauge automorphism.
\end{proposition}
\begin{proof}
With the above notations, suppose that $h(z)$ is a solution of \eqref{funct-eq-gauge} such that $h(1)=h(-1)=1$, which is not restrictive. Since $\Psi(i)=\Psi(e^{i\frac{\pi}{2}})=1$ and $h(i^2)=h(-1)=1$, we immediately see that $h(i)=1$. As $\Psi(z)=1$ for $0\leq \theta\leq\pi$, by using the functional equation we find that $h(e^{i\frac{\pi}{2^n}})=1$ for any $n\in \mathbb{N}$.
Consider then $z=e^{5i\frac{\pi} {4}}$. We have $\Psi(e^{5i\frac{\pi}{4}})=-1$ and $h(e^{2(5i\frac{\pi}{4})})=h(i)=1$, which in turn gives  $h(e^{5i\frac{\pi}{4}})=-1$. By induction we also see $h\big(e^{\frac{5\pi i}{2^{n+2}}}\big)=-1$. This proves that any solution $h$ of the functional equation 
with $f$ as in the statement cannot be continuous. 
\end{proof}
We can now devote ourselves to answering the question  whether $\out(\Q_2)$ is abelian. It turns out that it is not. Our strategy is merely to show that automorphisms $\beta^f$ corresponding to suitable functions $f$ do not commute in $\out(\Q_2)$ with the flip-flop. To begin with, if $\beta^f$ does commute with $\widetilde\lambda_f$ in $\out(\Q_2)$, then there must exist a unitary $V\in\Q_2$ such that $\widetilde\lambda_f\circ\beta^f\circ\widetilde\lambda_f=\textrm{ad}(V)\circ\beta^f$. Exactly as above, the unitary $V$ is then a continuous function of $U$, say $V=h(U)$. In addition, we also have
$$f(U^*)S_2=h(U)f(U)S_2h(U)^*=h(U)f(U)h(U^2)^*S_2$$ 
and so we find that $f$ and $h$ satisfy the equation $f(\bar{z})=h(z)f(z)\overline{h(z^2)}$ for every $z\in\mathbb{T}$, which can finally be rewritten as $f(\bar{z})\overline{f(z)}=h(z)\overline{h(z^2)}$, to be understood as an equation satisfied by the unknown function $h$, with $f$ being given instead. We next exhibit a wide range of continuous functions $f$ for which the corresponding $h$ does not exist. To state our result as clearly as possible, we fix some notation first. Let $f\in C(\mathbb{T},\mathbb{T})$ be such that $f(e^{i\frac{9\pi}{8}})=i$, and $f(z)=1$ everywhere apart from a sufficiently small neighborhood of $z=e^{i\frac{9\pi}{8}}$.   
\begin{proposition}
If $f\in C(\mathbb{T},\mathbb{T})$ is as above, then $\beta^f$ does not commute with the flip-flop in $\out(\Q_2)$. 
\end{proposition}
\begin{proof}
Repeat almost verbatim the same argument as in the foregoing proposition, now verifying that $h\big(e^{\frac{\pi i}{2^n}}\big)=1$ first and then 
$h\big( e^{\frac{9\pi i}{2^{n+3}}} \big)=-1$.
\end{proof}
Notably, this also yields the announced result on $\out(\Q_2)$.
\begin{theorem}\label{out-not-ab}
The group $\out(\Q_2)$ is not abelian.
\end{theorem}
Now we have got a better guess of what $\S_2$ might be, we can finally prove that it is in fact exhausted by isometries of the form $f(U)S_2$, where $f$ is a continuous function onto $\mathbb{T}$. This still requires some preliminary work. First observe
that given any $s\in\mathcal{S}_2$, a straightforward computation shows that both $s^* S_2$ and $s^*S_1$ commute with $U$, but then by maximality of $C^*(U)$ we can rewrite them as $h(U)$ and $g(U)$ respectively, with 
$h$ and $g$ being continuous functions.     
\begin{lemma}
There exists a continuous function $f$ such that $s=f(U)S_2$.
\end{lemma}
\begin{proof}
We start with the equality $s^*=s^*(S_1S_1^*+S_2S_2^*)=(s^*S_1)S_1^*+(s^*S_2)S_2^*$, in which we substitute the above expressions. This leads to 
$s^*=g(U)S_1^*+h(U)S_2^*$, that is $s=S_1g(U)^*+S_2h(U)^*=US_2g(U)^*+h(U^2)^*S_2=(Ug(U^2)^*+h(U^2)^*)S_2$. Therefore, our claim is true with
$f(z)=z\overline{g(z^*)}+\overline{h(z^2)}$.
\end{proof}
\begin{lemma}
With the notations set above, for every $z\in\mathbb{T}$ we have $|h(z)|^2+|g(z)|^2=1$.
\end{lemma}
\begin{proof}
It is enough to rewrite the equality $s^*s=1$ in terms of $h$ and $g$.
\end{proof}
\begin{lemma}
With the notations set above, for every $z\in\mathbb{T}$ we have $zh(z)\overline{g(z)}+g(z)\overline{h(z)}=0$.
\end{lemma}
\begin{proof}
Once again it is enough to rewrite the equality $s^*Us=0$, which is merely the orthogonality relation between $s$ and $Us$, in terms of $h$ and $g$.
\end{proof}
We are at last in a position to prove the main result on $\mathcal{S}_2$.
\begin{theorem}\label{Struct}
If $s\in\mathcal{S}_2$, then there exists a $f\in C(\mathbb{T},\mathbb{T})$ such that $s=f(U)S_2$.
\end{theorem}
\begin{proof}
At this stage, all we have to do is prove that $|f(z)|^2=1$. But 
$$|f(z)|^2=\left(\overline{g(z^2)}z+\overline{h(z^2)}(g(z^2)\overline{z}+h(z^2)\right)=1+z\overline{g(z^2)}h(z^2)+\bar{z}g(z^2)h(z^2)=1$$
\end{proof}

As an immediate consequence, we finally gain full information on $\Aut_{C^*(U)}(\Q_2)$.

\begin{theorem}\label{Isom}
The equalities hold $$\End_{C^*(U)}(\Q_2)=\Aut_{C^*(U)}(\Q_2)=\{\beta^f: f\in C(\mathbb{T},\mathbb{T})\}$$
In particular,  the semigroup $\End_{C^*(U)}(\Q_2)$ is actually a group isomorphic with $C(\mathbb{T},\mathbb{T})$. 
\end{theorem}
\begin{remark}
The bijective correspondence $f\leftrightarrow\beta^f$ is also a homeomorphism between $C(\mathbb{T},\mathbb{T})$ equipped with the uniform convergence topology and $\Aut_{C^*(U)}(\Q_2)$ endowed with the norm pointwise convergence.
\end{remark}
We end this section by proving that $\Aut_{C^*(U)}(\Q_2)$ is in addition a maximal abelian subgroup of $\Aut(\Q_2)$.
\begin{theorem}\label{MaxAb}
The group $\Aut_{C^*(U)}(\Q_2)$ is  maximal abelian in $\Aut(\Q_2)$.
\end{theorem}
\begin{proof}
We have to show that if $\alpha\in\Aut(\Q_2)$ commutes with any element of $\Aut_{C^*(U)}(\Q_2)$ then $\alpha$ is
itself an element of the latter group. Now, the equality $\alpha\circ\textrm{ad}(U)=\textrm{ad}(U)\circ\alpha$ gives
$\textrm{ad}(\alpha(U))=\textrm{ad}(U)$. Therefore, $\alpha(U)=zU$ for some $z\in\mathbb{T}$ by simplicity of $\Q_2$. The conclusion is then achieved if we show that
actually $z=1$. Exactly as above, we also have $\textrm{ad}(\alpha(g(U)))=\textrm{ad}(g(U))$ for any $g\in C(\mathbb{T},\mathbb{T})$. Again, thanks to simplicity we see that
$g(zU)=g(\alpha(U))=\alpha(g(U))=\lambda g(U)$ for some $\lambda\in\mathbb{T}$, possibly depending on $g$. In terms of functions we find the equality 
$g(zw)=\lambda_g g(w)$, which can hold true for any $g\in C(\mathbb{T},\mathbb{T})$ only if $z=1$. Indeed, when $z\neq 1$ is not a root of unity the characters $w^n$ are the sole eigenfunctions of the unitary operator
$\Phi_z$ acting on $L^2(\mathbb{T})$ as $(\Phi_z f)(w)\doteq f(zw)$. Finally, the case of a $z$ that is a root of unity is dealt with similarly. 
\end{proof}
\begin{remark}
The findings above are worth comparing with a result obtained in \cite{Cuntzsurvey} 
that the group of automorphisms of $\O_2$ fixing the diagonal $\D_2$ is maximal abelian too.
\end{remark}
Moreover, the theorem enables to thoroughly describe the automorphisms that send $U$ to its adjoint, which have been shown to be automatically outer.
\begin{theorem}
If $\alpha$ is an automorphism of $\Q_2$ such that $\alpha(U)=U^*$, then $\alpha(S_2)=f(U)S_2$ for a suitable $f\in C(\mathbb{T},\mathbb{T})$.
\end{theorem}
\begin{proof}
Just apply the former result to $\widetilde\lambda_f\circ\alpha$.
\end{proof}
Finally, the automorphisms $\beta^f$ can also be characterized in terms of the Cuntz-Takesaki generalized correspondence we discussed at the end of Section 4.1. Indeed, they turn out to be precisely  those $\Lambda\in\End(\Q_2)$ for which the corresponding $W\doteq U^*u^*\Lambda(U)u$ equals $1$, where $u$ stands for $u_\Lambda$ for brevity. For $W=1$ we find in fact the equality
$S_2uUu^*=U^2S_2=S_2U$, whence $uUu^*=U$. Therefore by maximality there exists a function $f\in C(\mathbb{T},\mathbb{T})$ such that $u=f(U)$, that is $\Lambda=\beta^f$. 

\subsection[Automorphisms $\alpha$ such that $\alpha(U)=zU$]{Automorphisms $\pmb{\alpha}$ such that $\pmb{\alpha(U)=zU}$}
The following discussion addresses the problem of studying those automorphisms $\Lambda$ of $\Q_2$ such that $\Lambda(U)=zU$, with $z\in\mathbb{T}$.  We start tackling the problem by defining two operators acting on $\ell_2(\mathbb{Z})$. The first is the isometry $S_z'$, which is given by $S_z'e_k\doteq z^ke_{2k}$.   The second is the unitary $U_z$, which is given by $U_ze_k\doteq z^k e_k$. The following commutation relations are both easily verified:
\begin{itemize}
\item $U_zU=zUU_z$
\item $U_zS_2=S_z'U_z$
\end{itemize}
The first relation can also be rewritten as $\textrm{ad}(U_z)(U)=zU$. We caution the reader that at this level $\textrm{ad}(U_z)$ makes sense as an automorphism of $B(\ell_2(\mathbb{Z}))$ only, because we do not know yet
whether $U_z$ sits in $\Q_2$. If it does, the first relation says, inter alia, that $\Q_2$ also contains a copy of the noncommutative torus $\mathcal{A}_z$ in a rather explicit way, which is worth mentioning.
In order to decide what values of $z$ do give a unitary $U_z$ belonging to $\Q_2$, the first thing to note is that if $U_z$ is in $\Q_2$, then it must be in the diagonal subalgebra $\D_2$, as shown in the following lemma.
\begin{lemma}
If $U_z$ is in $\Q_2$, then $U_z\in\D_2$.
\end{lemma}
\begin{proof}
A straightforward application of the equality $\D_2'\cap\Q_2=\D_2$, as $U_z$ is in $\D_2'=\ell_\infty(\mathbb{Z})$.
\end{proof}
The second thing to note is that the unitary representation $\mathbb{T}\ni z\mapsto U_z\in \U(B(\ell_2(\mathbb{Z})))$ is only strongly continuous. 
This implies that not every $U_z$ is an element of $\Q_2$. For the representation $z\mapsto U_z$ is only strongly continuous, which means 
the set $\{U_z\}_{z\in\mathbb{T}}$ is not separable with respect to the norm topology, whereas $\Q_2$ obviously is.
The next result provides a first answer to the question whether $U_z$ belongs to $\D_2$ . More than that, it also gives an explicit formula for $U_z$.
\begin{proposition}
If $z\in\mathbb{T}$ satisfies $z^{2^n}=1$ for some natural number $n$, then $U_z$ is in $\D_2$.
\end{proposition}
\begin{proof}
Obviously only primitive roots have to be dealt with. But for such roots, say $z=e^{i2\pi/2^k}$, the unitary $U_z$ may in fact be identified to the sum $\sum_{j=0}^{2^{k-1}}z^jP_j$, where the projection $P_j$ belongs to $\D_2$, being more explicitly given by
$P_{i_1i_2\ldots i_k}$, where the multi-index $(i_1,i_2,\ldots, i_k)\in \{1,2\}^k$ is the $j$-th with respect to the lexicographic order in which $2<1$ and the multi-index itself is read from right to left.
\end{proof}
The automorphisms obtained above are of course of finite order. More precisely, the order of $\textrm{ad}(U_z)$ is just the same as the order of the corresponding $z$. In other words, what we know is that the automorphism of
$C^*(U)$ induced by the rotation on $\mathbb{T}$ by a $2^n$-th root of unity extends to an inner automorphism of $\Q_2$, whose order is still finite being just $2^n$. Due to the lack of norm continuity of the representation $z\rightarrow U_z$, though, the case of a general $z$ is  out of the reach of the foregoing proposition and must needs be treated separately with different techniques. 
To begin with, we recall a result whose content should be well known. Nevertheless, we do include a proof not only for the sake of completeness but also to set some notations 
we shall need in the following considerations.
\begin{lemma}\label{lemma-proj}
Any projection $P\in \D_2$ is in the linear algebraic span of $\{S_\alpha S_\alpha^*\}_{\alpha \in W_2}$.
\end{lemma}
\begin{proof}
It is convenient to realize $\D_2$ as the concrete $C^*$-algebra $C(K)$, with the spectrum $K$ being given by the Tychonoff product $\{1,2\}^{\mathbb{N}}$. If we do so, the projections of $\D_2$ are immediately seen 
to identify with the characteristic functions of the clopens of $K$, and these are clearly the cylinder sets in the product space. The conclusion now follows noting that for any multi-index $\alpha\in W_2$ the characteristic function of a cylinder
$C_{\alpha}=\{x\in K: x(k)=\alpha_k\,\textrm{for any}\, k=0,1,\ldots ,|\alpha| \}$ corresponds indeed to $S_{\alpha}S_{\alpha}^*$. 
\end{proof}
At this point, it remains to show that $U_z$ does not belong to $\D_2$ for any other values of $z \in {\mathbb T}$. 
Although this could be done by means of explicit computations, as it was in an early version of this paper, we prefer to present a rather elegant method suggested by the referee.\\

The Cantor set $K=\{0,1\}^{\mathbb{N}}$ can also be realized as the ring of $2$-adic integer numbers $\mathbb{Z}_2$ via the bijective correspondence
$K\ni x=\{x_n\}_{n=0}^\infty\leftrightarrow \sum_{n=0}^\infty x_n2^n\in\mathbb{Z}_2$. In this picture the former digit $2$ has to be replaced by $0$. 
Accordingly, as of now we think of $\D_2$ as $C(\mathbb{Z}_2)$. Furthermore, as $\mathbb{Z}_2$ is
by definition the completion of $\mathbb{Z}$ under the metric induced by the $2$-adic absolute value, any $f\in C(\mathbb{Z}_2)$
is uniquely determined by its restriction to $\mathbb{Z}\subset\mathbb{Z}_2$. This gives an isometric inclusion of $\D_2\cong C(\mathbb{Z}_2)$ into 
$C(\mathbb{Z})\subset\ell_\infty (\mathbb{Z})$, which is nothing but the canonical representation of $\D_2$ on $\ell_2(\mathbb{Z})$. To see this,
it is enough to note that the generating projections $S_\alpha S_\alpha^*$ are indeed the characteristic functions of the subsets $\{2^n k+ l: k\in\mathbb{Z}\}\subset\mathbb{Z}$, where
$n=|\alpha|$ and $l=\sum_{j=0}^{n-1}\alpha_j2^j$. Phrased differently, we  have obtained the following useful characterization.

\begin{lemma}\label{Referee}
Let $f\in\ell_\infty(\mathbb{Z}) \subset B(\ell_2(\mathbb{Z}))$. Then $f$ is in $\D_2$ if and only if $f:\mathbb{Z}\rightarrow\mathbb{C}$ extends to a continuous function $\tilde f:\mathbb{Z}_2\rightarrow\mathbb{C}$.

\end{lemma}

We are now in a position to state and prove the main result of the present subsection.
\begin{theorem}
Let $z\in \mathbb{T}$. Then $U_z\in \D_2$ if and only if $z$ is a root of unity of order a power of $2$.
\end{theorem}
\begin{proof}
Thanks to Lemma \ref{Referee}, it is enough to make it plain when $\mathbb{Z}\ni k\mapsto z^k\in\mathbb{C}$ extends to a continuous function of
$\mathbb{Z}_2$. It is easily seen that this is the case if and only if $z$ is a dyadic root of unity.
\end{proof}

For  those $z\in\mathbb{T}$ such that $U_z$ lies in $\Q_2$ we can say a bit more.

\begin{proposition}
Let $z\in\mathbb{T}$ be a dyadic root of unity and let $\alpha\in \Aut(\Q_2)$ be such that $\alpha(U)=zU$. Then there exists a $f\in C(\mathbb{T},\mathbb{T})$ such that $\alpha(S_2)=f(zU)S_z'$.
\end{proposition}
\begin{proof}
By its very definition $\textrm{ad}(U_{z^{-1}})\circ\alpha(U)=U$. Therefore, we must have $\textrm{ad}(U_{z^{-1}})\circ\alpha=\beta^f$ for some $f\in C(\mathbb{T},\mathbb{T})$. But then $f(U)S_2=\beta^f(S_2)=U_{z^{-1}}\alpha(S_2)U_z$, i.e.
$\alpha(S_2)=U_zf(U)S_2U_{z^{-1}}=f(zU)U_zS_2U_{z^{-1}}=f(zU)S_z'$.
\end{proof}

\begin{remark}
We have already seen that if $U_z\in \Q_2$ then $S'_z\in \Q_2$. The converse, too, is  true. In fact, one can easily observe that $U_z=S^*_2S'_z$ hence the claim follows. In particular, whenever $U_z$ is not in $\Q_2$, the corresponding $\textrm{ad}(U_z)$ understood as an automorphism of the whole
$B(\ell_2(\mathbb{Z}))$ does not even leave $\Q_2$ globally invariant. 
\end{remark}

\appendix
\section{The functional equation $\pmb{f(z^2)=f(z)^2}$ on the torus}
This appendix presents a self-contained treatment of the functional equation $f(z^2)=f(z)^2$, of which we made an intensive use in the previous sections.
Although the following facts might all be well known, we do include complete arguments, because their proofs are not to be easily found in the literature, however carefully examined.  

\begin{proposition}
Let $f$ be a continuous function from $\mathbb{T}$ to $\mathbb{T}$ such that $f(z^2)=f(z)^2$ for every $z\in\mathbb{T}$. Then there exists a unique $n\in\mathbb{Z}$ such that $f(z)=z^n$.
\end{proposition}
\begin{proof}
Thanks to the compactness of $\mathbb{T}$ and the continuity of $f$, the winding number of $f$ is a well-defined integer $n\in\mathbb{Z}$, for details see e.g. Arveson's book \cite[Chapter 4, pp 114-115]{Arveson}.  The new function
$g(z)\doteq z^{-n}f(z)$ still satisfies our equation. Furthermore, the winding number of $g$ is zero by construction.
Therefore, there exists $h\in C(\mathbb{T},\mathbb{R})$ such that $g(z)=e^{2 \pi i h(z)}$ for every $z\in\mathbb{T}$. Rephrasing
the equation in terms of $h$, we find that $h(z^2)-2h(z)$ must be an integer for every $z\in\mathbb{T}$. 
By connectedness, the function $h$ is  thus a constant. Obviously there is no lack in generality if we also assume $h(z^2)-2h(z)=0$ for every $z\in\mathbb{T}$. Being bounded, the function $h$ is then forced to be identically zero, which finally leads to $f(z)=z^n$ for every $z\in\mathbb{T}$.
 \end{proof}

The above proposition can be regarded as a one-variable description of the characters of the one-dimensional torus.
It is worth pointing out, though, that it no longer holds true as soon as $\mathbb{T}$ is replaced by the additive group $\mathbb{R}$. In other words, there do exist continuous functions $f:\mathbb{R}\rightarrow \mathbb{T}$ such that
$f(2x)=f(x)^2$ other than $f_t(x)\doteq e^{itx}$, which are obtained by exponentiating non-linear continuous functions $g:\mathbb{R}\rightarrow\mathbb{R}$ such that $g(2x)=2g(x)$ for every $x\in\mathbb{R}$. 
However, any such $g$ cannot be everywhere differentiable with continuous derivative at $0$.\\  

The proof given above can be further simplified if we assume that $f$ satisfies a stronger functional equation, i.e. $f(z^n)=f(z)^n$ for every $n\in\mathbb{N}$ and $z\in\mathbb{T}$.
\begin{proposition}
If $f\in C(\mathbb{T}, \mathbb{T})$ satisfies $f(z^n)=f(z)^n$ for every $z\in\mathbb{T}$ and $n\in\mathbb{N}$, then there exists a unique $k\in\mathbb{Z}$ such that $f(z)=z^k$.
\end{proposition}
\begin{proof}
Obviously, it is enough to prove that $f(zw)=f(z)f(w)$ for any $z, w\in\mathbb{T}$. Let $z$ be a fixed element of $\mathbb{T}$, and let $C_z\subset\mathbb{T}$ be the set $C_z\doteq\{w\in\mathbb{T}: f(zw)=f(z)f(w)\}$.
From the equality $f(z^n)=f(z)^n$ we may note that $C_z$ contains the set $\{z^n:n=0,1,2,\ldots\}$. By continuity of $f$ we have in addition that $C_z$ is closed. Hence $C_z=\mathbb{T}$ if
$z=e^{2i \pi\theta}$, with $\theta$ being irrational. In other words, for such $z$, we have $f(zw)=f(z)f(w)$ for any $w\in\mathbb{T}$. The full conclusion is now easily got to by density. 
\end{proof}
\begin{remark}
The continuity assumption cannot be left out in either the above propositions. To see this, let $R\subset\mathbb{T}$ be the set $\cup_n H_n$, where $H_n\subset\mathbb{T}$
is the subgroup of the $n$th roots of unity. The function $f$ that is $1$ on $R$ and $f(z)=z$ on its complement still satisfies $f(z^n)=z^n$ for every $n\in\mathbb{N}$, as easily verified. Due to the density of $R$ in the torus, this function is nowhere continuous. Nevertheless, it is equal to the character $z$ almost everywhere. This seems to indicate that any measurable solution of the equation might equal a character almost everywhere. At any rate, it is worth pointing up that the solutions of
the equation $f(z^n)=f(z)^n$ do not enjoy automatic continuity, unlike the solutions of the equation $f(zw)=f(z)f(w)$, which are of course even automatically differentiable.
\end{remark}

\bigskip
\noindent {\it Acknowledgments} We would like to take this opportunity to thank L\'aszl\'o Zsid\'o for a fruitful conversation about the functional equation discussed in the Appendix. We are also grateful to Nicolai Stammeier for his valuable comments on the first draft of this paper. Finally, we owe the referee a debt of gratitude for his or her particularly attentive perusal of the manuscript, which resulted in many improvements not only to notation and presentation but even to the proof of the main result in the Appendix.

\end{document}